\documentclass[11pt]{article}

\usepackage{amsmath}
\usepackage{amssymb}
\usepackage{amsthm}

\usepackage{graphicx}

\newtheorem{lemma}{Lemma}
\newtheorem{definition}{Definition}
\newtheorem{theorem}{Theorem}

\newtheorem{proposition}{Proposition}
\newtheorem{rem}{Remark}
\begin{document}

\title{The Dynamical Discrete Web}

\author{
{\bf L.~R.~G.~Fontes}
{\small \sl Instituto de Matem\'{a}tica e Estat\'{i}stica,
Universidade de S\~{a}o Paulo, Brazil}\\
\and
{\bf C.~M.~Newman}
{\small \sl Courant Inst.~of Mathematical Sciences, NYU, New York, NY 10012}
\and
{\bf K.~Ravishankar}
{\small \sl Dept.~of Mathematics, SUNY College at New Paltz, New Paltz, NY 12561}
\and
{\bf E.~Schertzer}
{\small \sl Courant Inst.~of Mathematical Sciences, NYU, New York, NY 10012}
}

\date{}
\maketitle

\begin{abstract}
The dynamical discrete web (DDW),
introduced in recent work of Howitt and Warren,
is a system of coalescing simple symmetric
one-dimensional random walks which evolve in an extra continuous dynamical
time parameter $s$. The evolution is by independent updating of the underlying
Bernoulli variables indexed by discrete space-time that define the discrete
web at any fixed $s$. In this paper, we study the existence 
of exceptional (random) values of $s$ where the paths of the web
do not behave like usual random walks and the Hausdorff dimension
of the set of exceptional such $s$. 
Our results are motivated by
those about exceptional times for dynamical percolation
in high dimension by H\"{a}ggstrom, Peres and Steif, 
and in dimension two by Schramm and Steif. The exceptional behavior of
the walks in the DDW 
is rather different from the situation  
for the dynamical random walks of Benjamini,
H\"{a}ggstrom, Peres and Steif.
In particular, we prove that there are exceptional values of $s$ 
for which
the walk from the origin $S^s(n)$ has $\limsup S^s(n)/\sqrt{n} \leq K$
with a nontrivial dependence of the Hausdorff dimension on~$K$.
We also discuss how these and
other results extend to the dynamical Brownian web,
a natural scaling limit of the DDW.
The scaling limit is the focus of a paper in preparation;
it was also studied by Howitt and Warren and is related to the
Brownian net of Sun and Swart.
\end{abstract}

\newpage
\section{Introduction}
\label{intro}

In this paper, we present a number of results concerning a dynamical
version of coalescing random walks,
which was recently introduced in~\cite{HW07}.
Our results concern times of Hausdorff dimension less than one where
the system of coalescing walks behaves exceptionally. The results are analogous to
and were motivated by the model of dynamical percolation and
its exceptional times~\cite{HPS97,SS05}. In this section, we define
the basic model treated in this paper, which we call 
the dynamical discrete web (DDW),
recall some facts about dynamical percolation, and then briefly describe
our main results. The justification for calling this model a discrete web
is that there is a natural scaling limit, which is one of our main motivations
for analyzing the discrete web (as it is in~\cite{HW07}); we also
discuss in this section that scaling limit, which is a dynamical
version of the Brownian web (see~\cite{A81,TW98,STW00,FINR04}). 
A paper is in preparation~\cite{NRS07}
on the construction of that model, which is closely related to the Brownian
net of Sun and Swart~\cite{SS06}. We note that conjectures conerning ways
to construct scaling limits of dynamical percolation (in two-dimensional space)
appear in~\cite{CFN06}.
We further note that exceptional times for dynamical
versions of random walks in various spatial dimensions
have been studied in \cite{BHPS03,Hoff05,AH06} and elsewhere, but
these are quite different from the random walks of the DDW, as we note in Subsection~\ref{mainresults} below.

\subsection{Coalescing Random Walks And The Dynamical Discrete Web}
\label{coalescing}

Let $S^0 (t)$ for $t=1,2,\dots$ denote a simple symmetric random walk
on $\mathbb{Z}$ starting at $(0,0)$, i.e. at $0$ at $t=0$.
(For real $t \geq 0$, we set $S^0 (t) = S^0 ([t])$, where $[t]$ denotes
the integer part of $t$.)
If we also consider other simple symmetric random walks starting
from arbitrary points on the even space-time 
sublattice 
$\mathbb{Z}_{even}^2 = \{(i,j)\in \mathbb{Z}^2 : i+j 
\, \textrm{is even}\,\}$,
which are independent of each other except that they coalesce when they meet,
that is the system of (one-dimensional) coalescing random walks that is
closely related to the one-dimensional (discrete time) 
voter model (see~\cite{H78}) and may be thought of as a one plus one dimensional
directed percolation model.

The percolation structure is highlighted by defining $\xi_{i,j}^0$ for
$(i,j)\in \mathbb{Z}_{even}^2$ to be the increment 
between times $j$ and $j+1$ of the random walker
at location $i$ at time $j$. These Bernoulli variables
are symmetric and independent and the paths of all the coalescing random
walks can be reconstructed by assigning to any point $(i,j)$
an arrow pointing from  $(i,j)$ to $\{i+\xi_{i,j}^0,j+1\}$ and
considering
all the paths starting from arbitrary points in $\mathbb{Z}_{even}^2$
and following the arrows.
We note that there is also a set of dual (or backward) paths defined
by the same $\xi_{i,j}^0$'s with arrows from $(i,j+1)$ to $(i-\xi_{i,j}^0,j)$.
The collection of all dual paths is a system of backward (in time) coalescing
random walks that do not cross any of the forward paths.

The DDW is a very simple
stochastic process $W^s$ in a new dynamical time parameter $s$ whose
distribution at any deterministic $s$ is exactly that of the 
static coalescing
random walk model just described. Specifically, let
$(\xi)_s=(\xi_{i,j}^s,(i,j)\in \mathbb{Z}_{even}^2)_{s\in[0,\infty)}$
be a
family of independent continuous time cadlag Markov Processes
with state space $\{-1,+1\}$ and rate $\lambda /2$
for changing state in either direction, with the initial
condition that $(\xi_{i,j}^0,(i,j)\in \mathbb{Z}_{even}^2)$ is a family
of independent Bernoulli random variables with
$\mathbb{P}(\xi_{k,n}(0)=+1)=\frac{1}{2}$.

\subsection{Analogies With Dynamical Percolation}
\label{analogies}

Although this dynamical version of coalescing random walks sounds
quite trivial at first hearing, it turns out that it can have
interesting behavior at exceptional values of the dynamical time parameter s.
This is a feature that it shares in common with dynamical percolation.

Static percolation models are defined also in terms of independent Bernoulli
variables 
$\xi_z^0$, indexed by points $z$ in some $d$-dimensional lattice,
which in general are asymmetric with parameter $p$. There is a critical value
$p_c$ when the system has a transition from having infinite clusters (connected
components) with probability zero to having them with probability one. It is
expected that at $p=p_c$ there are no infinite clusters and this is proved for
$d=2$ and for high $d$ (see, e.g., \cite{G89}).
In dynamical percolation one
extends $\xi_z^0$ to time varying functions $\xi_z^s$, as in the case
of coalescing walks, except that the transition rates for the jump processes
$\xi_z^s$ are chosen to have the critical asymmetric $(p_c,1-p_c)$
distribution to be invariant. The question raised 
in~\cite{HPS97} 
was whether there were exceptional times when an infinite cluster (say,
one containing the origin) occurs, 
even though this does not occur at deterministic
times. This was answered negatively in~\cite{HPS97} for large $d$
and, more remarkably, was answered positively
by Schramm and Steiff for $d=2$ in~\cite{SS05},
where they further obtained upper and lower bounds on the Hausdorff
dimension (as a subset of the dynamical time axis) of these exceptional times.

\subsection{Main Results}
\label{mainresults}

We apply in this paper the approaches used for dynamical percolation
to the dynamical discrete web. Although we restrict attention to
one-dimensional random walks
whose paths are in two-dimensional
space-time and hence analogous to $d=2$ dynamical percolation,
by considering different possible exceptional phenomena, we use
both the high $d$ and $d=2$ approaches 
of~\cite{HPS97,SS05}.

A natural initial question was whether there might be exceptional
dynamical times $s$ for which 
the walk from the origin $S^s(t)$ is transient
(say to $+\infty$). Our first main result (see Theorem~\ref{th1}
in Section~\ref{tameness} below),
modeled after the
high-$d$ dynamical percolation results, is that
there are no such exceptional times.
As we explain in 
Remark~\ref{remtame} in Section~\ref{tameness}, a small modification 
of the proof of Theorem~\ref{th1}
shows that there are also no exceptional times where some pair
of walks avoids eventually coalescing.

Our other two main results are modelled after the $d=2$ dynamical
percolation results. One of them
(see Theorem~\ref{violation} below) concerns a kind of violation of
the Central Limit Theorem, or more accurately a kind of weak
subdiffusivity, by the random walk $S^s(t)$ for exceptional 
dynamical times
$s$; namely, that 
$S^s(t) \geq -k -K\sqrt{t}$ for all $t\geq0$. The
other (see Theorem~\ref{dimension}) gives upper and lower bounds
on the Hausdorff dimension of these exceptional times,
that depend nontrivially on the constant $K$ so that the
dimension tends to zero (respectively, one) as $K\to0$ (respectively,
$K\to \infty$).
This is strikingly in contrast with the dynamical random walks of
\cite{BHPS03} where there are no exceptional times for which
the law of the iterated logarithm fails.
To explain why the walks of \cite{BHPS03} can behave so differently
from those of the discrete web, we note that a single switch in the
former case affects only a single increment of the walk while some
switches in the discrete web change the path of the walker by a
``macroscopic'' amount, as discussed in the next subsection on scaling
limits~---~see also Figure~1  
where switching has changed one of the paths macroscopically.

By an obvious symmetry argument, there are also exceptional 
dynamical times $s$
for which $S^s(t) \leq k+K\sqrt{t}$. One may ask whether there are
exceptional $s$ for which $|S^s(t)| \leq k+K\sqrt{t}$.
As discussed in Remark~\ref{twosided} below, it can be shown,
at least for small $K$, that there are no such exceptional times.
The case of large $K$ is unresolved.

\begin{figure}
\label{fixeds}
\centering
\includegraphics[scale=0.6]{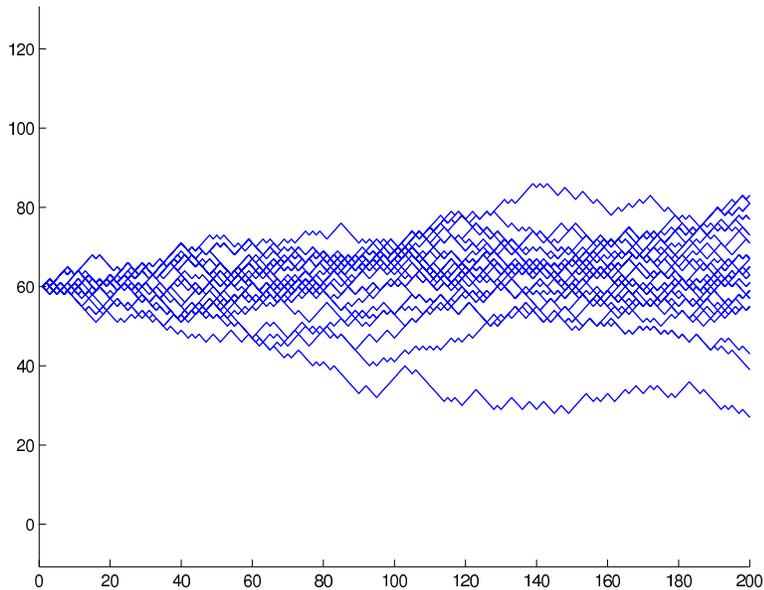}
\caption{
Let $S_{60}^s$ be the random walk at dynamical time $s$ starting from $x=60$.
This graph, with $t$ the horizontal and $x$ the vertical
coordinate, represents simultaneously the family of functions
$\{t\rightarrow S_{60}^s(t)\}_{s\in\mathbb{N}, 0\leq s\leq 40}$
($\lambda=\frac{1}{\sqrt{200}}, \, 0\leq t\leq 200$). The lowest path, 
say for $t$ greater than about $70$, differs ``macroscopically'' from the others.}
\end{figure}

\subsection{Scaling Limits}
\label{scalinglimits}

There is a natural scaling limit of the (static) coalescing random walks
model, the Brownian web (see~\cite{A81,TW98,STW00,FINR04}). 
Here one does a usual diffusive scaling
in which the random walk time $t$ is scaled by
$\delta^{-1}$, and space by $({\sqrt\delta})^{-1}$
so that the random walk path starting from 
$[x_0 ({\sqrt\delta})^{-1}]$ at time
$[t_0 \delta^{-1}]$ scales to a 
Brownian motion starting from $x_0$ at time $t_0$.
The collection of all random walk paths from all space-time starting
points scales to a collection of coalescing Brownian motion paths
starting from all points of continuum space-time.
Now taking the rate of switching to be of order $\sqrt{\delta}$,
rescaling time and space respectively by $\delta^{-1}$ and $({\sqrt\delta})^{-1}$,
and then letting $\delta$ go to $0$ leads to a nontrivial limit
$(\mathcal{W}^s)_{s\geq 0}$, the dynamical Brownian web.

The idea of taking a scaling limit of the dynamical discrete web to
obtain a dynamical continuum model is a natural one, which is at
the heart of~\cite{HW07}, although their approach appears to be somewhat different
than the one we had already been taking.
Both approaches are closely related to the Brownian net
construction of Sun and Swart~\cite{SS06} as will be
extensively explored in~\cite{NRS07}. As we shall discuss
in the next subsection, our approach is based on the construction of a certain
Poissonian marking of special space-time points of the (static) Brownian web.
These are the so-called $(1,2)$ points where a single Brownian web path
enters the point from earlier times and then two paths leave to
later times, one to the left and one to the right with exactly one of the
those two paths the continuation of the path from earlier time and the
other one ``newly-born''; see Figure~2.  

Neither the idea of doing a Poissonian marking of special points for
the Brownian web nor the idea of using those marked points to
construct a scaling limit of a dynamical discrete model is completely new.
In particular, we note that a different type of marking (of $(0,2)$ points)
was used in~\cite{FINR05} to study the scaling limits
of noisy voter models. Also the idea of using marked double points of
$SLE_6$ to construct the scaling limit of two-dimensional dynamical
percolation is discussed in~\cite{CFN06}.
Indeed, one motivation for the proposed marking in the $SLE_6$ context
was the analogy with markings of $(1,2)$ as well as of $(0,2)$ points
of the Brownian web.

In the dynamical Brownian web ${\cal W}^s$, one can also consider exceptional
dynamical times $s$ where the path ${\cal S}^s (t)$ starting from the
origin at continuous time $t=0$ behaves differently than an ordinary
Brownian motion path. The results of~\cite{NRS07}
are very similar to those of this paper for the discrete web. Indeed,
in some respects, the proofs are simpler since calculations with
Brownian motions are often easier than those with random walks. There
is however one substantial complication, which is the main focus
of~\cite{NRS07} and the reason we do not present the dynamical Brownian
web exceptional time results already in this paper. That complication
is the actual construction of the dynamical Brownian web --- a construction
that is considerably less trivial than that of the dynamical discrete
web, as we explain in the next subsection.

\subsection{The Dynamical Brownian Web}
\label{continuum}

\begin{figure}
\label{changedirections}
\centering
\includegraphics[scale=0.3]{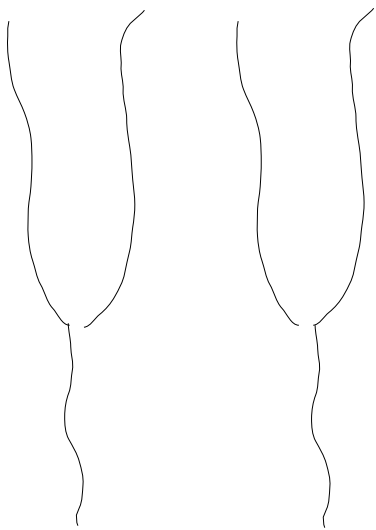}
\caption{
In this pair of diagrams, $t$ is the vertical and $x$ the
horizontal coordinate.
If a $(1,2)_l$ (left) point of the original 
web (left side of the figure) is marked to switch at
some $s_0 \in [0,s]$, then the direction of that $(1,2)$ point in
the web $\mathcal{W}^s$ is changed at $s=s_0$ so that it becomes
a $(1,2)_r$ (right) point as on the right side of the figure
with the incoming path joined to
the rightmost of the two paths starting from that point.}
\end{figure}

At the discrete level the scaling is chosen in such a way that between
the dynamical times $0$ and $s$, in a macroscopic box
(i.e., one with size of order $(\sqrt \delta)^{-1} \times \delta^{-1}$
in the original lattice), the number of
arrows that change direction will be of order $\delta^{-1}$.
The situation can be simplified
by focusing on switchings with
``macroscopic'' effects (i.e., switchings
that will lead to a macroscopic alteration of a walker's trajectory in
the initial web $W^0$). A priori, one should also consider
combinations of switchings that have macroscopic effects,
but it turns out
(this will be proved rigorously in~\cite{NRS07})
that the probability of macroscopic effects from switching
two or more arrows is negligible compared to switching
single arrows, and can be neglected. 

There is a natural way of characterizing those critical switchings. For
example, let us  consider the forward (rescaled) path $S^\delta$
starting from the origin in $W^0$
and assume that the arrow located at 
some $(S^\delta(t),t)$ is orginally oriented
to the left. Now we ask whether
a switching of this single arrow will alter the path in such
a way that the altered path  will be to the right of
$(S^\delta (t+\Delta t)+\Delta x,t+\Delta t)$, where  $(\Delta x,\Delta t)$
are both positive macroscopic quantities.
This will happen if and only if the backward path ${\hat S}^\delta$ starting
from $(S^\delta(t+\Delta t)+\Delta x,t+\Delta t)$  hits $S^\delta$ at time $t$
(more precisely, hits $(S^\delta(t), t+\delta)$ at time $t+\delta$). More
generally, the critical arrows leading to similar alterations are the
``contact'' points between $S^\delta$ and
the backward path ${\hat S}^\delta$ at which a
switching occurs on $[0,s]$. But it is now fairly easy to see
what the statistics of such a set of points are.
In fact, let
$m^\delta(t)$ be the random variable counting the number of such
switchings up to the macroscopic time $t$. The distribution of
$m^\delta(t)$ is simply given by:
\begin{eqnarray*}
m^\delta(t) & = & \sum_{i=1}^{\frac{1}{\sqrt\delta} L^\delta(t)} X_i^\delta
\ \ \ \textrm{with} \\
L^\delta(t) & = & \sqrt{\delta} \ \ \# \{k\leq \frac{t}{\delta}:
S^\delta(k \delta)={\hat S}^\delta (k \delta + \delta) \ \ \
\textrm{and there is a left arrow at $(S^\delta(k \delta), k \delta)$}\}
\end{eqnarray*}
where $\{X_i^\delta\}$ are i.i.d.~Bernoulli random variables with
$\mathbb{P}(\{X_i^\delta\}=1)= 1 - \exp{(-\sqrt{\delta} \lambda s})$,
which is $\sqrt{\delta} \ \lambda s + o(\sqrt{\delta})$ as $\delta \to 0$.

As $\delta \to 0$, $L^{\delta}$ converges to the ``local time''
$L$ of the forward Brownian path $B\in\mathcal{W}$ starting from the
origin along a backward Brownian path $\hat B$ starting on the right
of the path $B$ (the joint distribution of $B$ and $\hat B$ is analysed in~\cite{STW00} and this ``local time'' will be defined precisely in~\cite{NRS07}).
Further, it is a standard fact that
$t\rightarrow \sum_{i=1}^{t/ \sqrt\delta } X_i^\delta$ converges
to a Poisson process. Hence, $m^\delta(t)$ will converge to a Poisson
process run by the random clock $\lambda \  s L(t)$. In other words,
this set of points will consist of a two-dimensional ($t$ and $s$)
Poisson point
process with intensity measure $\lambda \ d  L \times l$, where $l$ is
Lebesgue measure and $d  L$ is the local time measure (note that
$\lambda  \ d  L \times l$ will be a locally finite measure
so that the Poisson process is
well defined).

So far, we have only selected the critical switchings inducing a
specific type of macroscopic effect. Namely, the ones altering the
path $B$ in such a way that a point originally on one side of
$B$ will be on the opposite side after switching occurs.
But in order to  select all the critical arrows leading to any kind of
macroscopic changes,  we should not only consider a Poisson process run
by the local time of a single forward path $B$ against a single backward path
$\hat B$,  but rather a Poisson process run by the
``local time of the entire forward web along the entire backward web''
multiplied by the intensity $\lambda s$.
In other words, the set of
marked points will be a three-dimensional Poisson point process with
intensity measure $\lambda \ {\cal L} \times l$, where $l$ is Lebesgue
measure (in the variable $s$) and ${\cal L}$ is
the local time measure of the forward web along the
backward web.
\footnote{Actually, the situation is a bit more complicated since
${\cal L}$ would not be a locally finite measure --- i.e., the set of
marked points in space-time is actually dense in $\mathbb{R}^2$.
However, like what is presented in~\cite{FINR04} (see p.~11 there),
one can add an extra coordinate and lift ${\cal L}$ to be a $\sigma$-finite
measure, or equivalently approximate ${\cal L}$ by a sequence
of locally finite measures ${\cal L}_n$, do the markings using
${\cal L}_n$, and then let $n \to \infty$.}

Since the $(1,2)$ points  of the  continuum web are
precisely those at which a
forward and a backward path meet (see, e.g.,~\cite{FINR04}),
the measure
$\lambda\  {\cal L} \times l$ will be supported by this set
of points. From our previous description of them, it should
be clear that each $(1,2)$ point
has a preferred left or right ``direction''. For example,
a $(1,2)_r$ (right)
point is one for which the continuing path (coming in from earlier
time) is to the right of
the ``newly-born'' path. Hence, at the continuum level, the analog of
an arrow switching will simply be a change of direction of all marked
$(1,2)$ points (see Figure~2). 
The web $\mathcal{W}^s$ at time $s_0$
will be ``simply''
deduced  from $\mathcal{W}^0$ by switching the direction of all
marked $(1,2)$ points whose $s$-coordinate is in $[0,s_0]$.

A last comment concerns the nature of the dependence of the
two continuum paths $B^s (t)$ and $B^{s'} (t)$.
These turn out to be a pair of ``sticky'' Brownian motions,
which are independent except when they touch each other. This is
one of the major observations in~\cite{HW07}; we give a brief
derivation of this fact in Section~\ref{pairsofpaths} by analyzing
pairs of paths in the discrete
setting to see what must occur in the continuum scaling limit.
In Section~\ref{tameness}, we state our main theorem about tameness;
i.e., that there are no exceptional dynamical times when the random walkers
are transient. 
We also give there some other results about tameness in two extended
remarks --- one about non-coalescence and the other about two-sided
bounds of order $\sqrt{t}$.
Then in Section~\ref{existence}, we show that there
are exceptional dynamical times when the walkers are 
(weakly) subdiffusive --- i.e., have one-sided bounds of order $\sqrt{t}$.
In Section~\ref{hausdorff} we derive upper and lower bounds
on the Hausdorff dimension of the set of such exceptional dynamical times.
Some estimates for random walks that are needed for our arguments
are given in Appendix~\ref{appendix}.

\section{Pairs Of Paths In The Dynamical Discrete Web}
\label{pairsofpaths}

%
%
%
\subsection{Interaction Between Paths In $W^s$ And $W^{s'}$}
\label{interaction}

The dynamics can be described, equivalently to the definition in
Section~\ref{intro}, in the following manner.
The initial configuration
is set to $\xi_0$, but now we place independent Poisson clocks
at each site $(i,j)$ that ring at rate $\lambda$.
Every time clocks ring we toss independent fair coins to decide
on the values of $\xi_{i,j}$ after the ring.
Statistically the two
descriptions are equivalent. 

The motivation for this second description is that it leads to a useful
representation of the interaction between the discrete webs at different
dynamical times $s$ and $s'>s$. In particular, let $S^s$ and $S^{s'}$ be the walks
starting at $(0,0)$ and defined for times $t=0,1,2,\dots$,
belonging  to $W^s$ and $W^{s'}$.

If $S^s(t)\neq S^{s'}(t)$, then $S^s(t+1)-S^s(t)$ and $S^{s'}(t+1)-S^{s'}(t)$
are independent since the directions of the arrows at two distinct sites
are independent. On the other hand, if $S^s(t)= S^{s'}(t)$, then the
next steps of the two walks are now correlated. If the clock at the site
$(S^s(t),t)$ did not ring on $[s,s')$, then the two paths will coincide
at time $t+1$. If it rang at least once, then with probability
$\frac{1}{2}$ they will coincide, and with probabily $\frac{1}{2}$ they won't.

Let us now define inductively a sequence of pairs of stopping times
$(\tau_i,\sigma_i)$ with $\tau_0=\sigma_0=0$ and:
\begin{eqnarray}
\tau_{i+1}=\inf\{t > \sigma_{i}: \ S^s(t)=S^{s'}(t)\} \\
\sigma_i=\inf\{t \geq \tau_{i}: \ \ \textrm{the clock at}\ (S^s(t),t)
\ \textrm{rings in} \  [s,s')\}
\end{eqnarray}
On the interval of integer time $[\tau_i,\sigma_i]$, the paths $S^s,S^{s'}$
coincide and at time $\sigma_{i}$ they decide to separate with
probability $\frac{1}{2}$. In other words, from time
$\sigma_i$, the walkers
$(S^s,S^{s'})$ move independently until the next meeting time $\tau_{i+1}$.
Hence, if we skip the intervals of time
$\{[\tau_i,\sigma_i)\}_i$, $(S^s, S^{s'})$ behave as two independent
random walks $(S_1,S_2)$, while if we skip the intervals
$\{[\sigma_i,\tau_{i+1})\}_i$, the two walks coincide with a single
random walk $S_3$.
Furthermore, since $S_3$ is constucted from the arrow configuration at
different sites than the ones used to construct $(S_1,S_2)$, it is
independent of $(S_1,S_2)$ ; and
$\{\sigma_i-\tau_{i}\}$ are i.i.d. random variables with
$\mathbb{P}(\sigma_i-\tau_{i} \geq k) = (e^{-\lambda |s-s'|})^k$.

Now, skipping the intervals $\{[\tau_i,\sigma_i)\}_i$ corresponds to
making the random time change
$t\rightarrow C(t)$  with $(C)^{-1}(t)={\hat L}(t)+t$, and
\begin{enumerate}
\item ${\hat L}(t)= \sum_{i=1}^{{\hat l}(t)} \sigma_i-\tau_{i}$ \, ,
\item ${\hat l}(t)=\#\{k\leq t: S_1(k)=S_2(k)\}$ \, ,
\end{enumerate}
while skipping $\{[\sigma_i,\tau_{i+1})\}_i$ corresponds to
making the time change $t\rightarrow t-C(t)$; i.e.,
\begin{eqnarray}
S^s(t)=S_1(C(t))+S_3(t-C(t)), \\
S^{s'}(t)=S_2(C(t))+S_3(t-C(t)),
\end{eqnarray}
where $(S_1,S_2,S_3)$ are three independent standard random walks.
In the following $(\tilde S_1,\tilde S_2)$,
distributed as $(S^s,S^{s'})$, will be referred to
as a pair of {\it sticky random walks\/}.

\begin{figure}
\label{varyings}
\centering
\includegraphics[scale=0.6]{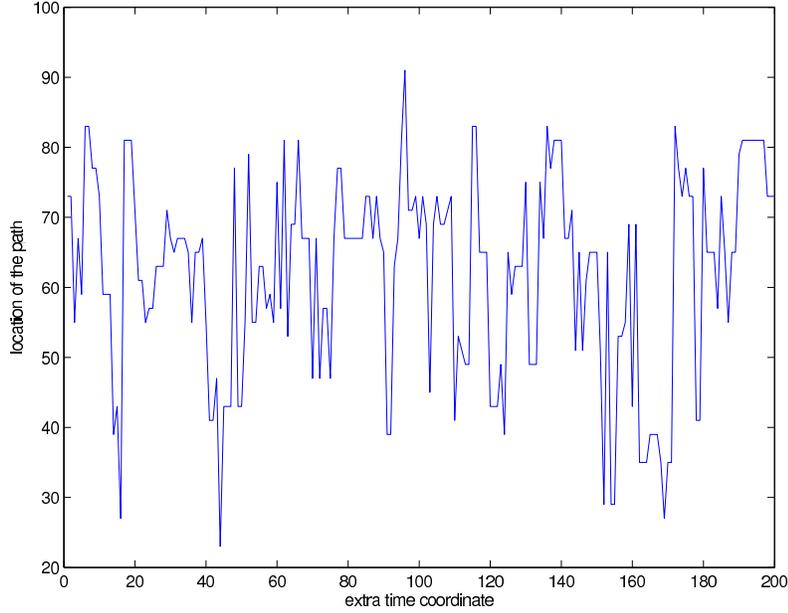}
\caption{ For $t$ fixed, this graph represents $s\rightarrow S^s(t)$
starting from $x=60$ (with $t=200$ and $\lambda=\frac{1}{\sqrt{200}}$).}
\end{figure}

\subsection{Sticky Paths In The Scaling Limit}
\label{stickypaths}
Time and space are respectively rescaled by
$\delta^{-1}$ and $\delta^{-1/2}$
and the rate of switching is taken as
$\lambda = {\bar \lambda} \sqrt{\delta}$.
From the previous section, the pair of rescaled processes
$(\sqrt{\delta} S^s(\frac{t}{\delta}), \sqrt{\delta} S^{s'}(\frac{t}{\delta}))$
is statistically equivalent to
\begin{equation}
\label{stat1}
\tilde S_1^{\delta}=S_1^\delta(C^\delta(t))+S_3^\delta(t-C^\delta(t)) 
\end{equation}
\begin{equation}
\label{stat2}
\tilde S_2^{\delta}=S_2^\delta(C^\delta(t))+S_3^\delta(t-C^\delta(t))
\end{equation}
where $S_1^\delta,S_2^\delta,S_3^\delta$ are three independent rescaled
random walks, and $(C^\delta)^{-1}(t)={\hat L}^\delta(t)+t$ with
\begin{enumerate}
\item ${\hat L}^\delta(t)={\delta}
\sum_{i=1}^{ {\hat l}^\delta(t) {\delta}^{-1/2}} T_i$
\item ${\hat l}^\delta(t)=\sqrt{\delta}\#\{k \leq t/{\delta}:
S_1^\delta(k\delta)=S_2^\delta(k\delta)\}$
\item $\{T_i\}$ are i.i.d random variables taking values
in $\mathbb{N}$, with $\mathbb{P}(T_i \geq k)=e^{-k \ \lambda |s-s'|}$.
\end{enumerate}

As $\delta \to 0$, $(S_1^\delta,S_2^\delta, S_3^\delta)$ converges in
distribution to three independent Brownian motions $(B_1,B_2,B_3)$; and
${\hat l}^\delta$ converges in distribution to the local time
${\tilde L}$ at the origin of $|B_1-B_2|$.
Moreover, as a consequence of the Law of Large Numbers,
if we take $\lambda=\bar \lambda \sqrt{\delta}$, with $\bar \lambda$
of order $1$, then ${\hat L}^\delta(t)$ converges to
$({|s-s'|\bar \lambda})^{-1} {\tilde L}(t)$.
Hence, it should come as no surprise
that $(\tilde S_1^{\delta},\tilde S_2^{\delta})$ converges to
\begin{eqnarray}
\tilde B_1=B_1(C(t))+B_3(t-C(t)) \\
\tilde B_2=B_2(C(t))+B_3(t-C(t))
\end{eqnarray}
where now $C^{-1}(t)=t+ ({|s-s'|\bar \lambda})^{-1} {\tilde L}(t)$,
which is identical in distribution to a pair of sticky Brownian motions
with stickiness parameter $({|s-s'|\bar \lambda})^{-1}$ (see, e.g., [SS06]).

We note that for small $\delta$, the location
${\tilde S}^{\delta,s}(t)$, of the path starting from
some $(x_0,t_0)$ is, for fixed $t$, quite discontinuous
in $s$ --- see Figure~3.  

\section{Tameness}
\label{tameness}

\begin{theorem}
\label{th1} Almost surely, all the paths are recurrent for every
$s$.
\end{theorem}

\begin{proof}
In Section~3 of~\cite{HPS97}, it is proved that for any
homogeneous graph with critical probability $p_c$ for percolation
and such that $\theta(p)$, the probability that (with parameter
$p$) the origin belongs to an infinite cluster, satisfies
$\theta(p)\leq C(p-p_c)$ for $p \geq p_c$, there is
almost surely no dynamical time $s$ at which percolation occurs.

In our setting, an entirely parallel argument can be used to show
tameness of the dynamical discrete web with respect to
recurrence. We discuss this briefly below, pointing to the
relevant parts of~\cite{HPS97}.

We consider the event $A_{i,j}$ that the walker starting from
$(i,j)$ does not visit the site to the left of its starting
position, that is, that the path 
$S_{(i,j)}$ started at $(i,j)$ does not
contain any $(i-1,k)$ with $k>j$. Let $\tilde\theta(p)$ be the
probability of that event under $p$
--- i.e., when the random walk increments are $+1$ (resp., $-1$)
with probability $p$ (resp., $1-p$). Under the usual coupled
construction of the model for $p\in[0,1]$, this event is
increasing with $p$ in $[1/2,1]$. For $p>\frac{1}{2}$ (resp.
$p<\frac{1}{2}$), 
$S_{(i,j)}$ is distributed as a right (resp. left)
drifting random walk. In particular, it is well known that for
$p\in[1/2,1]$
\begin{equation}\label{tame2}
\tilde\theta(p)=(2p-1)/p.
\end{equation}

We now describe the parallel argument alluded to above. Let
$\tilde N_{i,j}$ denote the cardinality of the set
$\{s\in[0,1]:\,A_{i,j}\mbox{ occurs in }(\xi)_s\}$.
$\tilde\theta(p)$ and $\tilde N_{i,j}$ are the analogues of
$\theta_v(p)$ and $N_v$ in Section~3 of~\cite{HPS97}. An analogue
of Lemma 3.1 there also holds here with the same proof, where here
$1/2$ is the analogue of $p_c$ there, and the analogue of
$N_{v,m}$ is the number $\tilde N_{i,j,m}$ of
$k\in\{1,2,\ldots,m\}$ such that $A_{i,j}$ occurs in
$$\bar\xi^{(k)}=\left\{\bar\xi^{(k)}_{i',j'}=\max_{s\in[(k-1)/m,k/m]}
\xi^s_{i',j'},\,(i',j')\in{\mathbb Z}_{even}\right\},$$ and we
conclude from~(\ref{tame2}) that ${\mathbb E}(\tilde
N_{i,j})<\infty$. Analogues of Lemmas 3.2 and 3.4 also hold with
the same proofs for
the analogue quantities, and with $A_{i,j}$ replacing the event \{{\em
v percolates}\}. We then have that ${\mathbb
E}(\tilde N_{i,j})=0$, and thus almost surely for every $s$ every
walker eventually visits the site to the left of its starting
position. The same is of course true of the site to the right of
the starting position by symmetry. We conclude that almost surely
for every $t$ every walker eventually visits every site in ${\mathbb Z}$
(infinitely often).
\end{proof}

\begin{rem}
\label{remtame} Another property of the static discrete web with
respect to which the dynamical one is tame is the almost sure
coalescence of all of its paths. It is enough to consider the case
of two paths. For those, a similar argument as that for recurrence
holds. The analogue objects to be considered in this case are as
follows. Given $v,v'\in{\mathbb Z}_{even}$ (let us assume that
$v_1<v'_1, v_2=v'_2$), let $C_{v,v'}$ be the event that the paths
starting from $v,v'$ do not eventually coalesce, and let $\tilde
N_{v,v'}$  be the cardinality of the set
$\{s\in[0,1]:\,C_{v,v'}\mbox{ occurs in }(\xi)_s\}$. For $1\leq
k\leq m$, let also $C_{v,v',k,m}$ be the event that the path of
$\bar\xi^{(k)}$ starting from $v'$ and that of
$\underline\xi^{(k)}$ starting from $v$ do not coalesce
eventually, where
$$\underline\xi^{(k)}=\left\{\underline\xi^{(k)}_{i,j}=\min_{s\in[(k-1)/m,k/m]}
\xi^s_{i,j},\,(i,j)\in{\mathbb Z}_{even}\right\}.$$ Let now
$\tilde N_{v,v',m}$ be the number of $k\in\{1,2,\ldots,m\}$ such
that $C_{v,v',k,m}$ occurs. To show that ${\mathbb E}(\tilde
N_{v,v'})=0$, we analyze $\tilde N_{v,v',m}$
and its related quantities analogously to the analysis of $\tilde N_{i,j,m}$
and its related quantities to show that ${\mathbb E}(\tilde
N_{i,j})=0$. In particular, the fact that ${\mathbb P}
(C_{v,v',k,m})\leq\mbox{const.}/m$ follows from standard
random walk estimates, using the fact that the difference of two
random walk paths is another random walk path.
\end{rem}

\begin{rem}
\label{twosided}
A main result of this paper is the existence of exceptional
$s$ such that for all~$t$, $S^s(t) \geq -k-K\sqrt{t} $
(see Theorem~\ref{violation} in Section~\ref{existence}),
and of course there are then also exceptional $s$ such that
$S^s(t) \leq k+K\sqrt{t} $. However, it
can be shown that there are no exceptional $s$
for the two-sided bound $|S^s(t)| \leq k+K\sqrt{t}$,
at least for small enough $K$. The precise condition 
on $K$ under which
we can prove this result is that $1-2p(K)\leq 1/2$, where $p(K)$
is defined in Proposition~\ref{upper} below. Note that this condition
implies according to Proposition~\ref{upper} that the Hausdorff
dimension of the set of 
exceptional $s$ for either of the corresponding one-sided
bounds does not exceed $1/2$. The proof of this tameness
claim combines
arguments like those of Theorem~\ref{th1} and Remark~\ref{remtame}
with the estimates of Lemma~\ref{cont} and Proposition~\ref{principal}
and with an application of the FKG inequalities. The specific FKG inequality,
for the two events $U_\epsilon^{\pm}$ that for some $s\in[0,\epsilon]$
and all~$t$,
$\pm S^s(t) \geq -k-K\sqrt{t} $, is that 
$\mathbb{P} (U_\epsilon^+ \cap U_\epsilon^-) \leq \mathbb{P} (U_\epsilon^+)
\, \cdot \, \mathbb{P} (U_\epsilon^-)$.
This is so because $U_\epsilon^+$ (resp., $U_\epsilon^-$) is an increasing
(resp., decreasing) event with respect to the basic $\xi_{(i,j)}^s$
processes --- see, e.g., Lemma~{3.3} of~\cite{HPS97} for more details.
We finally note that by essentially the same arguments one obtains
tameness for two-sided bounds of the form $-k_1-K_1\sqrt{t} \leq
S^s(t) \leq k_2+K_2\sqrt{t}$ provided that $K_1,K_2$ are small enough
that $(1-2p(K_1))+(1-2p(K_2))\leq 1$.
\end{rem}

\section{Existence of Exceptional Times}
\label{existence}

Let $\{d(k)\}_{k\geq0}$ be a sequence of  positive integers divisible by $4$.
We construct inductively a sequence of ``diffusive'' boxes $B_k$ in the following manner:
\begin{itemize}
\item $B_0$ is the rectangle with vertices $(-\frac{1}{2}\ d(0),0)$,
$(+\frac{1}{2}\ d(0),0)$, $(-\frac{1}{2}\ d(0),d(0)^2)$ and
$(+\frac{1}{2} d(0),d(0)^2)$.
\item Let $z_{n}=(x_n,t_n)$ and $z_n^{''}$ be respectively the middle
point of the lower edge and the upper right vertex of $B_n$. $B_{n+1}$
is the rectangle of height $d(n+1)^2$ and width $d(n+1)$ such that $z_{n+1}$
equals $z_n^{''}$ (see Figure~4).  
\end{itemize}
Let $A_k^s$ be the event that the path of $W^s$ starting at $z_k$ is 
at or to the right of 
$z_{k}^{''}$ at time $t_{k+1}$ and that it is never to the left of the
left edge $\partial_k$ of the box $B_k$ . We would like to prove that
for a certain choice of $\{d(k)\}$, there exist some exceptional times
$s$ at which $\{A_k^s\}$ occurs for every $k$. At those times,
this would imply that the path starting from the origin stays to the
right of the graphs obtained by patching together the left edges,
$\partial_k$, (see Figure~4).  
By the same kind of reasoning
used in dynamical 
percolation~\cite{SS05}, to prove that ordinary diffusive behavior
does not occur at certain exceptional dynamical times $s$,
it suffices to derive the following lemma, which we do later in this
section of the paper.
\begin{figure}
\centering
\includegraphics[scale=0.4]{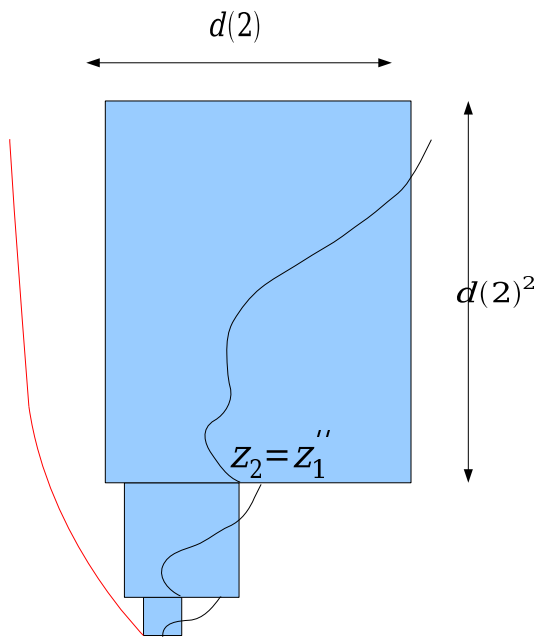}
\caption{
Construction of the first three boxes $(B_0,B_1,B_2)$
with $t$ the vertical and $x$ the horizontal coordinate. The solid curves
represent segments of the paths starting from 
$z_0$, $z_1=z_0^{''}$ and $z_2=z_1^{''}$ for which
the events $A_0^s$, $A_1^s$ and $A_2^s$ occur. The leftmost curve 
represents $-k-K\sqrt{t}$.}
\label{explain}
\end{figure}
\
\begin{lemma}
\label{propp}
There exists $\gamma_0>2$ such that if 
$d(k)=4 ([\frac{\gamma^k}{4\lambda}]+1)$ 
for $k\geq 0$ (where $[x]$ is the integer part of $x$) 
with $\gamma>\gamma_0$, then
\begin{equation}
\label{esd}
\inf_n \mathbb{P}(\int_0^1 \prod_{k=0}^n 1_{A_k^s} \ ds > \ 0 ) \geq p,
\end{equation}
where $p$ is bounded away from $0$ when
$\lambda$ is bounded away from $\infty$.
\end{lemma}
Let $E_n$ be the set of times $s$ on $[0,1]$ such that
$\bigcap_{k=0}^n A_k^s$ occurs. The previous lemma implies that
$\mathbb{P}(\bigcap_{n=0}^{\infty}( E_n\neq \emptyset)) \geq p >0$.
Since $\{E_n\}$
is obviously decreasing in $n$, if the $E_n$ were closed
subsets of $[0,1]$ it would follow that $\mathbb{P}(\bigcap_{n=0}^{\infty} E_n
\  \neq \emptyset) \geq p > 0$. As explained in the proof of the next theorem,
for $s \in \bigcap_{n=0}^{\infty} E_n$, $S^s (t) \geq -k -K \sqrt{t}$ for some
$k,K < \infty$ (depending on $\gamma$) and all $t = 0, 1, 2, \dots$.
Unfortunately, the set of times at which one arrow is
(or any finitely many are) oriented to the right (resp., to the left)
is not a closed subset of $[0,1]$ since we have a right continuous process,
and thus $E_n$ is not a closed set. This extra technicality is
handled like in Lemma 5.1 in~\cite{SS05},
as follows. Let $\hat {\cal S}$ denote the (random)
set of all switching times for all $\xi_{i,j}^s$'s.
By modifying every $\xi_{i,j}^s$ so that for $s' \in \hat {\cal S}$,
$\xi_{i,j}^{s'} = +1$ (rather than being right-continuous),
each $E_n$ is replaced by a closed ${\bar E}_n\supseteq E_n$.
On the other hand, $\cap_{n=1}^\infty {\bar E}_n = \cap_{n=1}^\infty E_n$
as a consequence of the fact that
$\hat {\cal S}$ is countable and by independence of the $\xi_{i,j}$'s
no $s' \in \hat {\cal S}$ can be exceptional.

The uniformity with respect to small $\lambda$ in Lemma~\ref{propp} means that once
space and time are diffusively rescaled by $\delta^{-\frac{1}{2}}$
and $\delta^{-1}$ and $\lambda$ is rescaled by $\delta^{\frac{1}{2}}$,
the inequality~(\ref{esd}) is still valid, with $p$
fixed as
$\delta \to 0$.

As a consequence of Lemma~\ref{propp}, we will obtain the following.
\begin{theorem}({\bf violation of the CLT})
\label{violation} For
${\bar \lambda},\delta \in (0,\infty)$, let 
$S^s_\delta(t)=S_{\delta,\bar\lambda}^s(t)=\sqrt\delta S^s([t / \delta])$ 
where $S^s(\cdot)=S^s_{1,\lambda} (\cdot)$ is the path starting at $(0,0)$ 
of the dynamical discrete web with switching rate 
$\lambda=\bar\lambda\sqrt{\delta}$.
There exists $K < \infty$ such that $p_{\delta, \bar\lambda}(K,\bar k)$,
the probability to have a nonempty set of exceptional times s in $[0,1]$
for which $S^s_{\delta,\bar \lambda} (t) \geq -\bar k -K\sqrt{t}$ 
for all $t\geq 0$ satisfies the following:
\begin{enumerate}
\item For any $\bar k > 0$, there exist 
$\bar \lambda_0,\lambda_0 \in (0,\infty)$ such that
\begin{equation}
\label{gg}
\inf_{\bar \lambda\geq \bar \lambda_0, \ 
\delta \leq ({\lambda_0} / {\bar \lambda_0})^2} 
p_{\delta, \bar\lambda}(K,\bar k)>0 \, .
\end{equation}
\item Similarly, for any $\bar \lambda_0$, $\lambda_0 \in(0,\infty)$, 
there exists $\bar k<\infty$ such that (\ref{gg}) is valid.
\item For any fixed $\delta$, $\bar \lambda\in(0,\infty)$, 
$p_{\delta,\bar \lambda}(K,0)>0$.
\end{enumerate}
\end{theorem}

\begin{proof}
In the unrescaled coordinates, we take boxes $B_k$ as in Lemma~\ref{propp}
with $d(k)=4([\frac{\gamma^k}{4\lambda}]+1) \in 
(\gamma^k/\lambda,4+\gamma^k/\lambda]$.
Then in rescaled coordinates we have
boxes ${\bar B}_k$ with (spatial) width 
${\bar d}(k) = (\lambda/ \bar\lambda) d(k) 
\in(\gamma^k/\bar\lambda,4\sqrt\delta+\gamma^k/\bar\lambda]$ 
and (temporal) height ${\bar d}(k)^2$.
Let ${\bar \partial}$ denote the right-continuous function obtained by joining
together the left boundaries ${\bar \partial}_k$ of ${\bar B}_k$.
On $[{\bar t}_n,{\bar t}_{n+1})$ with
${\bar t}_n = \bar d(0)^2+\bar d(1)^2+\dots \bar d({n-1})^2$,
we have
${\bar \partial}(t) = {\bar \partial}({\bar t}_n) =
{\bar x}_n - (1/2) {\bar d}(n)$ $=
(\bar d(0)+\bar d(1)+\dots+\bar d(n-1)-\bar d(n))/2$. 
If $K,\bar k$ are such that
\begin{equation}
\label{boundary}
{\bar \partial}({\bar t}_n) \ \geq \ - (\bar k+K \sqrt{{\bar t}_n}) \
\textrm{for} \  n = 0, 1, 2, \dots \ ,
\end{equation}
then we will have ${\bar \partial}(t) \geq -(\bar k+K \sqrt{t})$
for all $t \geq 0$ as desired. 

The inequality (\ref{boundary}) can be rewritten as 
\begin{equation}
\label{gggg}
\bar d(n)\leq 2\bar k+\bar d(0)+\cdots+
\bar d(n-1)+2K[{\bar d(0)}^2+\cdots+{\bar d(n-1)}^2]^{1/2} \, . 
\end{equation}
Using the bound $\bar d(n)\leq 4\sqrt{\delta}+\gamma^n/\bar\lambda$ on the 
left-hand side of (\ref{gggg}) and the bounds 
$\bar d(j)\geq\gamma^j/\bar\lambda$ on the right-hand 
side, it follows that in order to verify~(\ref{gggg})
it suffices to have, for $n=0,1,2,\dots$,
\begin{equation}
\label{anotherineq}
\gamma^n\leq (2 \bar k-4\sqrt{\delta})\bar\lambda +
\frac{\gamma^n-1}{\gamma-1}+2K\sqrt{\frac{\gamma^{2n}-1}{\gamma^2-1}}\ .
\end{equation}
Using the elementary bound
$\sqrt{\gamma^{2n} - 1} \geq \gamma^n(1 - \gamma^{-2n})$ (for $\gamma \geq 1$),
we see that in order to verify~(\ref{anotherineq}),
it suffices to have, for $n=0,1,2,\dots$,
\begin{equation}
\gamma^n(\frac{\gamma-2}{\gamma-1} - \frac{2K}{\sqrt{\gamma^2 -1}})
\leq (2 \bar k-4\sqrt{\delta})\bar\lambda -
\frac{1}{\gamma-1}- \frac{2K}{\sqrt{\gamma^2-1}}\gamma^{-n} \, .
\end{equation}
Choosing $K=(\frac{\gamma-2}{2}) \sqrt{\frac{\gamma+1}{\gamma-1}}$ 
yields this inequality provided $(2\bar k-4\sqrt{\delta})\bar\lambda-1\geq0$.
It is easy to see that for any $\bar k>0$, this will be valid provided 
$\lambda_0$ is small enough and $\bar \lambda_0$ is large enough so that 
$\bar k\geq 2\lambda_0/\bar\lambda_0+1/(2\bar\lambda_0)$. 
This and Lemma~\ref{propp} prove 
the first claim of the theorem; the second claim, in which $\bar \lambda_0$ 
and $\lambda_0$ are given, follows similarly.

We now turn to the proof of the final claim. We set $\delta=1$
since essentially the same proof works for any $\delta>0$. Let 
$T_m^{[0,1]}$ 
denote the set of $s\in [0,1]$ such that 
$S^s(n)\geq-m-K\sqrt{n}$ for $n \geq 0$ and 
let $j$ be an integer so large that 
(by the second claim of the theorem) $p_{1,\bar \lambda}(K,j)>0$. First, 
$T_0^{[0,1]} \supset \hat T_{j}^{[0,1]}\bigcap
\{s\in [0,1]:\xi_{m,m}^s =+1 \ \ \textrm{for} \ \ m<j\}$ 
where $\hat T_{j}^{[0,1]}$ is the set of $s\in [0,1]$ 
such that $S^s_{(j,j)}(n)\geq-K\sqrt{n}$ 
for $n\geq j$. 
Furthermore, $\hat T_{j}^{[0,1]} \supset \bar T_{j}^{[0,1]}$, 
where $\bar T_{j}^{[0,1]}$ is the 
set of $s\in [0,1]$ such that 
$S^s_{(j,j)}(n)-j\geq-j -K \sqrt{n-j}$ for $n\geq j$. 
But $\bar T_{j}^{[0,1]}$ is just the translation (from $(0,0)$ to $(j,j)$) of 
$T_{j}^{[0,1]}$. Since $\{s\in [0,1]: \forall \  k<j, \ \xi_{k,k}^s=+1\}$ and 
$\bar T_j^{[0,1]}$
are independent, it follows that
\begin{equation}
p_{1,\bar\lambda}(K,0)\ \geq \ p_{1,\bar\lambda}(K,j) \ \ 
\mathbb{P} (\forall \ s\in[0,1], \ \forall \  k<n, \ \ \xi_{k,k}^s=+1)>0 \, .
\end{equation}
\end{proof}

In particular, if we assume that $s\rightarrow W_{\delta}^s$ converges
to the dynamical Brownian web (see Subsection~\ref{continuum})
in some appropriate sense as $\delta \to 0$, this shows
that the analogue of Theorem~\ref{tameness} (except for the final claim with $\delta$ fixed and $\bar k=0$) will be valid for the continuum
model as well.

\bigskip

We now turn to:

\noindent {\it Proof of Lemma~\ref{propp}.} By the
Cauchy-Schwarz inequality,
\begin{eqnarray}
\mathbb{P}(\int_0^1 \prod_{k=0}^n 1_{A_k^s} \ ds>0) & \geq &
\frac{\left(\mathbb{E}\left[\int_0^1\prod_{k=0}^n 1_{A_k^s} \ ds\right]\right)^2}
{\mathbb{E}\left[\left(\int_0^1 \prod_{k=0}^n 1_{A_k^s} \ ds \right)^2\right]} \\
& = &\left(\left[\int_0^1 \int_0^1 \prod_{k=0}^n
\frac{\mathbb{P}(A_k^s \bigcap A_k^{s'})}{\mathbb{P}(A_k)^2} \ ds \ ds' \right]
\right)^{-1} \label{integrand}
\end{eqnarray}
where $A_k = A_k^0$ and the equality is a consequence of the stationarity
of $s\rightarrow W^s$ and the independence between the different boxes $B_k$.
It is enough to show that the integrand in
the last expression of~(\ref{integrand}) is bounded above by a integrable 
function  on $[0,1]\times[0,1]$, uniformly in $n$. The rest of the proof 
will verify this property.

Now, for fixed
$k$ and two deterministic times $(s,s')$, let us rescale space and time
respectively by $\delta^{-1/2} ={d(k)}$ and $\delta^{-1}=d(k)^2$.
Also, let $\tilde S_1^{\delta},\tilde S_2^{\delta}$ be the paths starting at
$1/2$ at time $0$, defined as the rescaled and translated version of the
paths $(S_1,S_2)\in(W^s,W^{s'})$ starting at $z_k$, the middle point of the
lower segment of the box $B_k$.  $(\tilde S_1^{\delta},\tilde S_2^{\delta})$
is a pair of sticky rescaled random walks starting at $1/2$ at time $t=0$
whose  statistics (up to a translation of starting point)
are described in Equations~(\ref{stat1})-(\ref{stat2}).

By definition,
$\mathbb{P}(A_k^s\cap A_k^{s'})=\mathbb{P}(\textrm{for} \ \ i=1,2,
\ \ \tilde S_i^\delta(1)>1 \ \ \textrm{and} \ \
\inf_{t \in [0,1]} \ \tilde S_i^\delta (t) >0)$.
To complete the proof of Lemma~\ref{propp}, we
will use the following lemma, in which $\delta^{-1/2}$ 
may be taken as an integer divisible by $4$.

\begin{lemma}
\label{inequality}
Let $\tilde S_1^{\delta},\tilde S_2^{\delta}$ be a pair of sticky
random walks starting from $1/2$ at time 
$t=0$ as defined in (\ref{stat1})-(\ref{stat2}).
Let $A_i = A_i(\delta)$ be the event that for $i=1,2$,
$\tilde S_i^\delta(1)\geq1$ and $\inf_{t \in [0,1]} \ \tilde S_i^\delta (t) \geq0$.
If $\lambda \ |s-s'| \leq 1$, then for $\frac{\sqrt\delta}{\lambda \ |s-s'|}$
small enough,
\begin{equation}
\label{lemmainequality}
\mathbb{P}(A_1(\delta)\cap A_2(\delta))\leq \mathbb{P}(A_1(\delta))
\ \mathbb{P}(A_2(\delta)) + K' \left(\frac{\sqrt\delta }{\lambda \ |s-s'|}\right)^{a}
\end{equation}
where $K'$ and $a$ are positive constants 
(independent of $\lambda,\, s,\, s'$ and $\delta$).
\end{lemma}

\begin{proof}
In this proof we set $\Delta=\frac{\sqrt\delta}{\lambda \ |s-s'|}$.
For any positive $\alpha$,
let $S_i^{\delta}$ be as in (\ref{stat1})-(\ref{stat2}) and let
$\inf S_i^\delta\equiv \inf_{t\in [0,1]} S_i^\delta(t)$. Then
\begin{align}
\label{equation1}
\mathbb{P}(A_1\cap A_2)& \leq
\mathbb{P}(\, {\textrm for} \  i=1,2, \ \ \ S_i^\delta(1) \geq 1-\Delta^{\alpha},
\ \ \ \inf S_i^\delta \geq -\Delta^{\alpha}) & \, \nonumber \\
&  +\sum_{i=1}^2 \{ \mathbb{P}(A_1\cap A_2,S_i^\delta(1)<1-\Delta^{\alpha})
+ \mathbb{P}(A_1\cap A_2, \, \inf S_i^\delta <-\Delta^{\alpha}) \} \, . &
\end{align}

We start by
dealing with the first term of the right-hand side of~(\ref{equation1}).
First,
\begin{eqnarray}
\label{90}
\mathbb{P}(\, {\textrm for} \   i=1,2, \ \ \ S_i^\delta(1) \geq 1-\Delta^{\alpha},
\ \ \
\inf S_i^\delta \geq -\Delta^{\alpha})
\leq \mathbb{P}(A_1(\delta))\mathbb{P}(A_2(\delta)) \nonumber \\
+2 \mathbb{P}(S_1^\delta(1)\in[1-\Delta^\alpha,1])
+2 \mathbb{P}(\inf S_1^\delta\in[-\Delta^\alpha,0])
\end{eqnarray}
using the independence of the walks $S_i^\delta$
and the equidistribution of ${\tilde S}_1^\delta, {\tilde S}_2^\delta,
S_1^\delta, S_2^\delta$.
The last two terms can be dealt 
with in a number of ways. For example, in~\cite{F73}, it is proved that a 
sequence of  rescaled standard random walks $\{S_i^{\delta}\}_{\delta}$ 
and a Brownian Motion $\hat B$ can be constructed on the same probability
space in such way that for for any $a<\frac{1}{4}$ the quantity 
$\mathbb{P}(\sup|\hat B-S^\delta|>\delta^a)$ goes to $0$ faster than any 
power of $\delta$. On this probability space,
\begin{eqnarray}
\label{91}
\mathbb{P}(S_i^\delta(1)\in[1-\Delta^\alpha,1])
\leq \mathbb{P}( {\hat B}(1) \in [1- 2 \Delta^\alpha,1+\Delta^\alpha])+
\mathbb{P}(\sup|\hat B-S_i^\delta|>\Delta^\alpha) \, , \\
\mathbb{P}(\inf S_i^\delta\in[-\Delta^\alpha,0])\leq \mathbb{P}(
\inf {\hat B} \in [- 2 \Delta^\alpha,\Delta^\alpha])+ 
\mathbb{P}(\sup|\hat B-S_i^\delta|>\Delta^\alpha) \, , \label{es2}
\end{eqnarray}
where the sup (and inf) are over $t \in [0,1]$. 
Since $\lambda|s-s'|\leq1$, we have $\sqrt{\delta} 
\leq \Delta$, implying that for $\alpha<\frac{1}{2}$ and $\delta$ small 
enough the last terms on the right-hand side
of (\ref{91}) and (\ref{es2}) are bounded by 
$O(\sqrt{\delta})$, and consequently by $O(\Delta)$. Finally, 
(\ref{90}), (\ref{91}) and (\ref{es2}) yield:
\begin{equation}
\mathbb{P}(\, {\textrm for} \   i=1,2, \ \ \ S_i^\delta(1) \geq 1-\Delta^{\alpha},
\ \inf S_i^\delta \geq -\Delta^{\alpha})\leq
\mathbb{P}(A_1) \ \mathbb{P}(A_2) + K' \ \Delta^\alpha
\end{equation}
where $K'$ is a positive constant and $\alpha<\frac{1}{2}$.

It only remains to deal with the rest of the terms on the
right-hand side of Equation~(\ref{equation1}). We will prove that
$\mathbb{P}(A_1(\delta)\cap A_2(\delta),S_1^\delta(1)<1-\Delta^{\alpha})
\leq K'' \Delta^{a'}$;
the other terms can be treated in a similar fashion.

For any $\beta>0$, we have
\begin{eqnarray}
\mathbb{P}(A_1(\delta)\cap A_2(\delta),S_1^\delta(1)<1-\Delta^{\alpha})
\leq
\mathbb{P}(\tilde S_1^\delta(1)\geq 1,S_1^\delta(1)<1-\Delta^{\alpha},
{\hat L}^\delta(1)\leq {\Delta^{\beta}})
\nonumber \\
+\mathbb{P}({\hat L}^
\delta(1)>{\Delta^{\beta}})
\end{eqnarray}
Lemma~\ref{lemm1} below takes care of the second term on the right-hand side
of the inequality when $0<\beta< 1$.
On the other hand, since
\begin{equation}
\tilde S_1^\delta(t)=S_1^\delta (t)+
(S_1^\delta(C^\delta(t))-S_1^\delta(t))+S_3^\delta(t-C^\delta(t)),
\end{equation}
we have that
\begin{eqnarray}
\mathbb{P}(\tilde S_1^\delta(1)\geq 1, S_1^\delta(1)<1-\Delta^{\alpha},
{\hat L}^\delta(1)\leq{\Delta^{\beta}}) \leq  \ \ \ \ \ \nonumber \\
\mathbb{P}(|S^\delta_3 (1-C^\delta(1))| \geq \frac{\Delta^\alpha}{2},
{\hat L}^\delta(1)\leq{\Delta^{\beta}})
+\mathbb{P}(|S_1^\delta(1)-S_1^\delta(C^\delta(1))|
\geq \frac{\Delta^\alpha}{2}, {\hat L}^\delta(1)\leq {\Delta^{\beta}})\, .
\end{eqnarray}
Now, on the event
$\{{\hat L}^\delta(1) \leq {\Delta^{\beta}}\}$,
by definition of $C^\delta(t)$, we have for any $t'\in[0,1]$:
\begin{equation}
(C^\delta)^{-1}(t')\leq t' +\Delta^{\beta}.
\end{equation}
Since $C^\delta(t)\leq t$ and $C^\delta$ is an increasing function, it follows that
\begin{equation}
C^\delta(t)\geq t-\Delta^{\beta}
\end{equation}
implying
\begin{eqnarray}
\mathbb{P}(\tilde S_1^\delta(1)\geq 1 ,S_1^\delta(1)<1-\Delta^{\alpha},
{\hat L}^\delta(1)\leq{\Delta^{\beta}}) \leq \\
\mathbb{P}(\sup_{t \in [0,\Delta^{\beta}]}|S_3^\delta (t) |\geq 
\frac{\Delta^\alpha}{2})
+\mathbb{P}(\sup_{t \in [1-\Delta^{\beta},1]}|S_1^\delta(1)-S_1^\delta(t)|
\geq\frac{\Delta^{\alpha}}{2}) \, .
\end{eqnarray}

By (the $L^2$ version of) Doob's inequality, we then have
\begin{equation}
\mathbb{P}(\tilde S_1^\delta(1)\geq 1 ,S_1^\delta(1)<1-\Delta^{\alpha},
{\hat L}^\delta(1)\leq{\Delta^{\beta}}) \leq \bar K \Delta^{\beta- 2 \alpha }.
\end{equation}
Therefore, taking $\alpha<\beta / 2$ with $0<\beta<1$ gives the desired
bound for the second term of the right-hand side of inequality~(\ref{equation1}).
For the first term of this inequality, we only needed 
 $\alpha\in(0,1/2)$ and
the conclusion follows.
\end{proof}

\begin{lemma}
\label{lemm1}
For any $1 >\beta>0$ and
$\Delta=\frac{\sqrt \delta}{\lambda |s-s'|}$ small enough
\begin{equation}
\label{lemma3bound}
\mathbb{P}({\hat L}^\delta(1) \geq {\Delta^\beta})\leq
\tilde K \Delta^{1-\beta}
\end{equation}
where $\tilde K >0$.
\end{lemma}

\begin{proof}
By the Markov inequality,
\begin{eqnarray}
\mathbb{P}({\hat L}^{\delta}(1)\geq{\Delta^{\beta}})\leq \delta^{\frac{1}{2}(1-\beta)}
\mathbb{E}(T_1)\mathbb{E}({\hat l}^{\delta}(1)) |s-s'|^\beta \lambda^\beta \\
\textrm{with $\mathbb{E}(T_1)=\sum_{k=1}^\infty e^{-\lambda \ |s-s'| \ k}=
\frac{\exp(-\lambda |s-s'|)}{1-\exp(-\lambda |s-s'|)}$} \ .
\end{eqnarray}
Since ${\hat l}^{\delta}(1)$ converges in distribution to the local time
of $\sqrt{2} B$, where $B$ is a standard Brownian motion,
$\mathbb{E}({\hat l}^{\delta}(1))$ is uniformly bounded in $\delta$. Furthermore,
$E(T_1)=0(\lambda |s-s'|)^{-1}$, implying that
\begin{equation}
\mathbb{P}({\hat L}^{\delta}(1)\geq{\Delta^{\beta}})\leq \tilde K 
\left(\frac{\sqrt \delta}{\lambda |s-s'|}\right)^{1-\beta}.
\end{equation}


\end{proof}

\noindent {\bf Completion of proof of Lemma~\ref{propp}.} \\
Recall that 
$d(k)= \ 4([\frac{\gamma^k}{4 \lambda}]+1) \geq \gamma^k / \lambda$,
where $\gamma>\gamma_0>2$ with $\gamma_0$
to be fixed later. By Lemma~\ref{inequality},
there exists $m$ small enough such that (\ref{lemmainequality}) is
valid for $\frac{\sqrt{\delta}}{\lambda \ |s-s'|}\leq m$. We define
$N_0 = [\frac{-\log(m|s-s'|)}{\log{\ \gamma}}]$
so that for $k>N_0$, (\ref{lemmainequality}) is
valid for $A_k^s$ and $A_k^{s'}$. $N_0$ is independent of~$\lambda$
and  since 
$m\geq (\gamma^{N_0+1}\ |s-s'|)^{-1}\ $,  
\begin{eqnarray}
\label{final1}
\prod_{k=N_0+1}^{\infty} \left(
\frac{\mathbb{P}(A_k^s \cap A_k^{s'})}{\mathbb{P}(A_k)^2} \right) & \leq &
\prod_{k=N_0+1}^\infty (1+ \frac{K' / \mathbb{P}(A_k)^2 }
{ |s-s'|^a\gamma^{a(N_0+1)} \ \ \gamma^{a(k-N_0-1)}}) \nonumber \\
& \leq &\prod_{k=N_0+1}^{\infty} (1+ \frac{K' m^a}{\inf _n\mathbb{P}(A_n)^2}
\ \ \frac{1}{\gamma^{a(k-N_0-1)}})
\end{eqnarray}
where 
$a$ and $K'$ are as in Lemma~\ref{inequality}. The right-hand side of
(\ref{final1}) is independent of $\lambda$ and
$|s-s'|$ and is finite. Indeed, $0<\inf_n P(A_n)$  since the
boxes $B_k$ have diffusively scaled  sizes and therefore 
$\mathbb{P}(A_k^s)\rightarrow\mathbb{P}(A)$ as $k \to \infty$, where $A$ is the
event that a Brownian motion
${\hat B}(t)$ starting at $\frac{1}{2}$ at
time 0 has ${\hat B}(1)>1$ and $\inf_{t\in[0,1]} {\hat B}(t)>0$.

On the other hand,
\begin{eqnarray}
\label{final2}
\prod_{k=0}^{N_0} \frac{\mathbb{P}(A_k^s \cap A_k^{s'})}
{\mathbb{P}(A_k)^2} & \leq & (\sup_k \frac{1}{\mathbb{P}(A_k)})^{N_0} \nonumber \\
& \leq & \exp(\frac{\log{\sup_k
\frac{1}{\mathbb{P}(A_k)}}}{\log{\gamma}}\log(\frac{1}{m |s-s'|}))=
\frac{1}{m^b|s-s'|^b} \, ,
\end{eqnarray}
where $b=\frac{\log{\sup_k (1/ \mathbb{P}(A_k))}}{\log{\gamma}}\,$.

Taking $\gamma>\sup_k \frac{1}{\mathbb{P}(A_k)}$,
by (\ref{final1}) and (\ref{final2}) we have that for every $n$ 
\begin{eqnarray}
\label{energy}
 \prod_{k=0}^n
\frac{\mathbb{P}(A_k^s \cap A_k^{s'})}{\mathbb{P}(A_k)^2} 
\leq {\tilde K}' \frac{1}{|s-s'|^b} \, ,  
\end{eqnarray}
with ${\tilde K}'>0$ and $b<1$. 
Since $(s,s')\rightarrow |s-s'|^{-b} \in L^1([0,1]\times[0,1])$) 
and (\ref{energy}) is uniform in
$n$, this concludes the proof of Lemma~\ref{propp}. 

\section{Hausdorff Dimension Of Exceptional Times}
\label{hausdorff}
In this section, we derive some lower and upper bounds for the set of
exceptional dynamical times $s \in [0,\infty)$. To simplify notation,
the rate of switching $\lambda$ and the scaling parameter $\delta$
will both be taken equal to $1$ from now on. However, as in the previous
section, it can easily be checked that essentially all the results stated below are
again uniform in $\delta \leq 1$ once space and time  and $\lambda$ are
properly rescaled according to $\delta$ (see Subsection~\ref{scalinglimits}). The result that is not uniform as stated is Proposition \ref{constant}; to have uniformity, $k\geq0$ should be replaced by $k\geq k_0>0$ for any $k_0>0$.

\begin{definition}
\label{def1}
We say that $s$ is a $K$-exceptional time if the path
$(S^s(t):\, 0\leq t < \infty)$ in $W^s$ starting from the origin
at time $t=0$ does not cross the moving boundary $t\rightarrow -1-K\sqrt{t}$.
$\mathcal{T}(K)$ is then defined as the set of all $K$-exceptional times
$s\in [0,\infty)$.
\end{definition}

Clearly, the set consisting of all the K-exceptional times in $[0,\infty)$
is a non-decreasing function of $K$.
Note that in Definition~\ref{def1}, the constant term for the moving boundary
is fixed at $1$. The next propostion asserts that for fixed $K$ the Hausdorff
dimension $dim_H$ of the set of exceptional times is unchanged if $1$ is replaced
by any $k \geq 1$. (The remark following the proof of the proposition
points out that more can be proved by essentially the same arguments.)
We note that as in dynamical percolation (see Sec.~6 of~\cite{HPS97}),
$dim_H(\mathcal{T}(K))$ is a.s. a constant by the ergodicity in $s$ of
the dynamical discrete web.

\begin{proposition}
\label{constant}
The Hausdorff dimension $dim_H$ of the set $T_k = T_k(K)$
of exceptional times $s \geq0$ such
that $S^s$ does not cross the moving boundary $-k-K\sqrt{n}$
does not depend on $k \geq 0$ (for fixed $K$).
\end{proposition}

\begin{proof}
By monotonicity in $k$, it is enough to prove that
$dim_H(T_{k}(K))\leq dim_H(T_{0}(K))$ for $k$ any 
positive integer. By the same reasoning used to prove the last claim of
Theorem \ref{violation}, the Hausdorff dimension of $T_0(K)$ is~$\geq$ 
the Hausdorff dimension of the set of times $\{s\geq 0: \forall 
\  m<k, \ \xi_{m,m}^s=+1\} \bigcap \bar T_k(K)$ where $\bar T_k(K)$ is 
the translation (from $(0,0)$ to $(k,k)$) of $T_k(K)$. 
By ergodicity in $s$, the a.s. constant $dim_H(\bar T_{k}(K))$
is the essential supremum of the random variable 
$dim_H(\bar T_k(K)\bigcap[0,1])$. On the other hand, since 
$\bar T_k(K)\bigcap[0,1]$ and $\{s\in[0,1]: \forall \  m<k, \ \xi_{m,m}^s=+1\}$ 
are independent and the probability to have 
$\{\forall \ s\in[0,1], \ \forall \  m<k, \ \xi_{m,m}^s=+1\}$ is strictly
positive, it follows that 
$dim_H(\{s \in [0,1]: \forall \  m<k, \ \xi_{m,m}^s=+1\} \bigcap \bar T_k(K))$ 
has the same essential sup as $dim_H(\bar T_k(K) \bigcap [0,1])$. Hence
$dim_H(T_{k}(K)) = dim_H(\bar T_k(K)) \leq dim_H(T_{0}(K))$
and the conclusion follows.


\end{proof}

\begin{rem} 
Define ${\tilde {\mathcal{T}}}(K)$ to be the set
of $s$ such that for some $j_0\geq 0$ and some $i_0 \in \mathbb{Z}$,
the infimum over $n \geq 0$ of $\{S_{(i_0,j_0)}(n) + K \sqrt{n-j_0}\}$
(or equivalently of $\{S_{(i_0,j_0)}(n) + K \sqrt{n}\}$)
is $> \, -\infty$.
It is not hard to see by arguments like those of
Proposition~\ref{constant} that $dim_H({\tilde {\mathcal{T}}}(K))
= dim_H(\mathcal{T}(K))$.
\end{rem}

\subsection{Lower Bound}
\label{lowerbound}

\begin{proposition}
\label{lower}
$\textrm{dim}_{H}(\mathcal{T}(K))$ converges to $1$ as $K\to \infty$.
\end{proposition}

\begin{proof}
Let $\alpha <1$ be fixed. Since $\mathcal{T}(K)$ increases with $K$ it
is enough to show that for $K$ large enough the Hausdorff dimension
is at least $\alpha$.

Consider the random measure $\sigma_n$, defined as
$\sigma_n(E)=\int_E \prod_{k=0}^n (1_{A_k^s}/ \mathbb{P}(A_k)) \ ds$,
for any Borel set $E$ in $[0,1]$ ( where $\{A_k^s\}$ are defined as in
Section~\ref{existence}). We define the $\alpha$-energy of $\sigma_n$
as
\begin{equation}
\mathcal{E}_{\alpha}(\sigma_n)=\int_{0}^1 \int_0^1 \frac{1}{|s-s'|^{\alpha}} \
d \sigma_n(s) \ d \sigma_n(s') \, .
\end{equation}
By identical arguments as in Section~6 of [SS05], if the expected value
of $\mathcal{E}_{\alpha}(\sigma_n)$ is bounded above as $n \to \infty$,
then the Hausdorff dimension of the set of exceptional $s$ (in $[0,\infty)$)
for which $\cap_{k=0}^\infty  A_k^s$ occurs is at least $\alpha$.
By Fubini's Theorem, 
\begin{equation}
\mathbb{E}(\mathcal{E}_{\alpha}(\sigma_n))=
\int_0^1 \int_0^1 |s-s'|^{-\alpha} \ \prod_{k=0}^n
\frac{\mathbb{P}(A_k^s\cap A_k^{s'})}{\mathbb{P}(A_k^s)^2}
\ \ ds \ ds'
\end{equation}
and 
by (\ref{energy}), we have that
\begin{equation}
\label{Haus}
\sup_n \mathbb{E}(\mathcal{E}_{\alpha}(\sigma_n))
\ \leq \ \bar K \left(\left[\int_0^1 \int_0^1
\frac{1}{|s-s'|^{b+\alpha}} \ ds \ ds' \right] \right) 
\end{equation}
with $b=(\log{\sup_k \frac{1}{\mathbb{P}(A_k)}})/\log{\gamma}$.
In particular, taking
$K=(\frac{\gamma -2}{2})\sqrt{\frac{\gamma+1}{\gamma-1}}$
as in the proof of Theorem~\ref{violation} with $\gamma$ large enough,
$b+\alpha$ can be made smaller than $1$,
and the right-hand side of (\ref{Haus}) is finite.
\end{proof}

\begin{rem}Combining the proofs of Lemma~\ref{propp}
and Proposition~\ref{lower} with the remark following
Proposition~\ref{constant}, we can obtain a more explicit lower bound
on $dim_H({\mathcal T} (K)) = dim_H( {\hat {\mathcal T}}(K))$ as follows.
Set ${\bar \gamma}_0 = 1/{\mathbb{P}(A)}$, where $A$ is the event
that a Brownian motion ${\hat B}(t)$ starting at $1/2$
at time $t=0$ has ${\hat B}(1)>1$ and $\inf_{t\in [0,1)}{\hat B}(t)>0$,
and define ${\bar \gamma}(K)$ as the solution in $(2,\infty)$
of $K(\gamma)=(\frac{\gamma -2}{2})\sqrt{\frac{\gamma+1}{\gamma-1}}$
for $K>0$. Letting $K_0=K({\bar \gamma}_0)$, we then have
\begin{equation}
\label{lowerdim}
dim_H({\mathcal T}(K)) \ \geq \ 1\, - \,
\frac{\log{\bar \gamma_0}}{\log{\bar \gamma(K)}}
\ \textrm{for} \ K\, > \, K_0 \, .
\end{equation}
\end{rem}

\subsection{Upper Bound}
\label{upperbound}
We will prove the following proposition.
\begin{proposition}
\label{upper}
For any $0<l<1$,
$\textrm{dim}_H(\mathcal{T}(K)) \leq 2(\frac{1}{2}-p(\frac{K}{l}))$
where $2p({K}) \in (0,1)$ is the real solution
$u = u(K) \in (0,1)$ of the equation
\begin{equation}
\label{sato}
f(u,{K}) \equiv
\frac{\sin(\pi u / 2) \Gamma(1+ u / 2)}{\pi} \sum_{n=1}^{\infty}
\frac{(\sqrt{2} K)^n}{n !} \ \Gamma((n-u)/2)=1 \, .
\end{equation}
Furthermore,
$\lim_{K\uparrow\infty}2(\frac{1}{2}-p(K))=1$ and
more significantly
$\lim_{K\downarrow0}2(\frac{1}{2}-p(K))=0$.
\end{proposition}
As a consequence of Propositions \ref{constant}, \ref{lower} and \ref{upper},
we immediately have the following.
\begin{theorem}
\label{dimension}
The limits as $K\to 0$ and $K \to \infty$ of $dim_H(\mathcal{T}(K))$ are
\begin{equation}
\lim_{K\uparrow\infty} \ \textrm{dim}_H(\mathcal{T}(K))=1,
\ \  \lim_{K\downarrow0} \ \textrm{dim}_H(\mathcal{T}(K))=0 \, .
\end{equation}
For any continuous function $g$ starting at $(-k,0)$, $k>0$ such that
$\lim_{t\uparrow\infty}\frac{g(t)}{\sqrt{t}}=0$, the set of exceptional times
for which the path starting from the origin at time $0$
does not cross $g$ has Hausdorff dimension zero.
\end{theorem}

To prove Proposition~\ref{upper}
we need the two following lemmas.

\begin{lemma}(Sato [S77])
\label{samere}
Let $\tau=\inf\{t>0 :B(t)=-k+ K\sqrt{t}\}$, where $k,K$ are both positive,
and $B$ is a standard Brownian motion. Then there exists
$q \in (0,\infty)$ such that
\begin{equation}
\lim_{t \to \infty} t^{p(K)} \mathbb{P}(\tau>t) = \ q \, ,
\end{equation}
where $2 p({K})$ is the real solution in $(0,1)$ of~(\ref{sato}).
\end{lemma}

\begin{lemma}
\label{cont}
For any $K,k>0$, let
$\tau_\epsilon=\inf\{t>0 :B_{\epsilon}(t)=-k-K \sqrt{t}\}$
where
$B_{\epsilon}(t)= B(t)+ 2 \epsilon t$ and $B$ is a standard
Brownian motion. Then for some $C=C(k,K)<\infty$ and $\epsilon \leq 1$,
\begin{equation}
\label{exitbound}
\mathbb{P}(\tau_\epsilon=\infty) 
\leq C \, \epsilon^{1-2(\frac{1}{2}-p({K}))} \, .
\end{equation}
\end{lemma}

\begin{proof}
Let $f_{\epsilon}$ be the density of $\tau_{\epsilon}$.
By the Girsanov Theorem,
\begin{equation}
f_{\epsilon}(t)=\ \exp(-2(k+K\sqrt{t})\epsilon-2\epsilon^2 t)\ f(t)
\end{equation}
where $f=f_{0}$ is the density corresponding to a
standard Brownian motion. Therefore, since
$\mathbb{P}(\tau_0 <\infty)=1$, we have
\begin{equation}
\mathbb{P}(\tau_\epsilon=\infty)=
\int_0^\infty (1-e^{(-2k-2K\sqrt{t})
\epsilon-2\epsilon^2 t} ) \ f(t) \ dt  \, .
\end{equation}
Integrating by parts, we get that
\begin{eqnarray}
\mathbb{P}(\tau_{\epsilon}=\infty)=(1-e^{-2k\epsilon})+
\int_0^{\infty} \left(\frac{\epsilon K}{\sqrt{t}}+2\epsilon^2 \right)
e^{-2k\epsilon-2K\epsilon\sqrt{t}-2\epsilon^2 t}
\mathbb{P}(\tau\geq t) dt  \nonumber \\
\leq  2 \epsilon k + \epsilon K \int_0^{\infty}
\frac{e^{-2\epsilon K \sqrt {t}}}{\sqrt{t}} \mathbb{P}(\tau\geq t)
+2 \epsilon^2 \int_0^{\infty} e^{-2K\epsilon\sqrt{t}}
\mathbb{P}(\tau\geq t) \ dt \, .
\end{eqnarray}
On the one hand, by Lemma~\ref{samere},
\begin{eqnarray}
\epsilon K \int_0^{\infty}  \frac{1}{\sqrt{t}}
e^{-2 \epsilon K \sqrt {t}} \mathbb{P}(\tau\geq t) \ dt
& \leq & C_1 (K) \, \epsilon\int_0^\infty
\frac{e^{-2\epsilon K \sqrt {t}}}{t^{p+\frac{1}{2}}} \ dt \nonumber \\
& =  & C_2 (K) \, \epsilon^{1-2(\frac{1}{2}-p)} \int_0^\infty
\frac{e^{- \sqrt{u}}}{u^{p+1/2}} \ du  \nonumber \\
& =  & C_3 (K) \, \epsilon^{1-2(\frac{1}{2}-p)} \, .
\end{eqnarray}
On the other hand,
\begin{eqnarray}
\epsilon^2 \int_0^{\infty} e^{-2K\sqrt{t}\epsilon}
\mathbb{P}(\tau \geq t) \ dt &
\leq & C_4 (K) \, \epsilon^2 \int_0^\infty  e^{-2K \epsilon\sqrt{t} }
\frac{1}{t^p}\ dt \nonumber \\
& \leq & C_5 (K) \, \epsilon^2 \, \epsilon^{2(p-1)}
\int_0^{\infty} \frac{e^{-v}}{v^{2p-1}} dv \nonumber \\
& = & C_6 (K)  \, \epsilon^{1-2(\frac{1}{2}-p)} \, .
\end{eqnarray}
The last three displayed equations together easily imply~(\ref{exitbound}).
\end{proof}

\noindent{ \it Proof of Proposition~\ref{upper}.}
We are now ready to obtain an upper bound for the Hausdorff dimension of the
set of K-exceptional times $\mathcal{T}(K)$.
Let us partition $[0,1]$ into intervals of equal length
$\epsilon$, and select the intervals containing a K-exceptional time. The
union of those is a cover of $\mathcal{T}(K)$ and we now estimate
the number $n(\epsilon)$ of intervals in the cover.

Let $U_{\epsilon}$ be the event that there is a time $s$ in $[0,\epsilon]$
such that  $s\in\mathcal{T}(K)$. From the full dynamical arrow configuration
for all $ s \in [0,\epsilon]$, we construct a static arrow configuration
as follows. We declare 
the static arrow at $(i,j)$ to be 
right oriented if and only the dynamical arrow is right oriented
(i.e., $\xi_{i,j}^s = +1$) at some 
$s\in[0,\epsilon]$ (a similar construction was used in Section~\ref{tameness}).
In this configuration, the path $S_{\epsilon}$ starting
from the origin and following the arrows is a slightly right-drifting
random walk with 
$\mathbb{P}(S_{\epsilon}(n+1)-S_{\epsilon}(n)=+1)=
\frac{1}{2}+\frac{1}{2}(1-e^{-\epsilon})$. Clearly,
\begin{equation}
\mathbb{P}(U_{\epsilon})\leq
\mathbb{P}(\forall n, \, S_{\epsilon}(n)\geq -1-K \sqrt{n}) \, .
\end{equation}
Proposition~\ref{principal} of Appendix~\ref{appendix}
implies that for any $l<1$
\begin{equation}
\mathbb{P}(\forall n, \, S_{\epsilon}(n)\geq -1-K \sqrt{n})
\leq C_7 (K,l) \, \mathbb{P}(\forall t>0, \,
B_{\epsilon}(t)\geq -3-\frac{K}{l} \sqrt{t})+ O (\epsilon)
\end{equation}
and by Lemma~\ref{cont} it follows that
\begin{equation}
\mathbb{P}(U_{\epsilon}) = O(\epsilon^{1-2(\frac{1}{2}-p(\frac{K}{l}))}) \, .
\end{equation}
Hence
\begin{equation}
\mathbb{E}(n(\epsilon)) = O ( \epsilon^{-2(\frac{1}{2}-p(\frac{K}{l}))}) 
\end{equation}
so that
\begin{equation}
\limsup_{\epsilon \to 0} \mathbb{E}(\frac{n(\epsilon)}
{\epsilon^{-(1-2p(K/l))}}) \ < \ \infty \, .
\end{equation}
By Fatou's Lemma,
$\liminf_{\epsilon \to 0} n(\epsilon) \,
\epsilon^{1-2p(K/l)}$
is almost surely bounded, which implies that the Hausdorff dimension of
$\mathcal{T}(K)$ is bounded above by $2(\frac{1}{2}-p(\frac{K}{l}))$
and completes the proof of Proposition~\ref{upper}.

\appendix

\section{Some Estimates For Random Walks}
\label{appendix}
We will prove the following proposition.
\begin{proposition}
\label{principal}
Let $B_{\epsilon}(t)= B(t) + 2 \epsilon \, t $, where $B$ is
a standard Brownian motion, and let
$S_{\epsilon}$ be a discrete time simple random walk with drift given by
\begin{equation} \mathbb{P}(S_{\epsilon}(n+1)-S_{\epsilon}(n)=1)=
\frac{1}{2}+ \frac{1}{2} (1-e^{-\epsilon}) \, .
\end{equation}
For $K>0$ there exists $C>0$ such that for any $0<l<1$,
\begin{equation}
\label{theineq}
\mathbb{P}(\forall n\in\mathbb{N}, \, S_{\epsilon}(n) \geq-1-l K \sqrt{n})
\leq C \ \mathbb{P}(\forall t\in\mathbb{R}^+, \,
B_{\epsilon}(t) \geq-3- K \sqrt{t} ) + O(\epsilon) \, .
\end{equation}
\end{proposition}

We consider $S'_{\epsilon}$ the discrete time random walk embedded in
the drifting Brownian motion $B_{\epsilon}$. Namely,
we define inductively a sequence of stopping times
$T_i^\epsilon$ and their increments $\{\tau_i^\epsilon =
T_i^\epsilon - T_{i-1}^\epsilon\}$, with $T_{0}^\epsilon=0$ and
\begin{equation}
\label{tautau}
T_{n+1}^\epsilon=\inf\{t>T_n^\epsilon :  \
|B_{\epsilon}(t)- B_{\epsilon}(T_n^\epsilon)| \geq 1\}
\end{equation}
and then we define
$S'_\epsilon(n)=B_{\epsilon}(T_n^\epsilon)$. The proof of
Proposition~\ref{principal} will be done by coupling $S'_{\epsilon}$
and $B_{\epsilon}$ in this particular way. Note that $S'_{\epsilon}$ has
a different drift than $S_\epsilon$ since
$\mathbb{P}(S'_{\epsilon}(n+1)-S'_{\epsilon}(n)=1)$ is not
$\frac{1}{2}+\frac{1}{2}(1-e^{-\epsilon})
\approx\frac{1}{2}+\frac{1}{2}\epsilon$, but rather is
$(e^{4\epsilon}-1) / (e^{4\epsilon}-e^{-4\epsilon})
\approx\frac{1}{2}+\epsilon$. But proving (\ref{theineq}) with
$S_{\epsilon}$ replaced by $S_{\epsilon}'$ suffices since $S_{\epsilon}'$
has a larger positive drift than $S_{\epsilon}$.

Now, let us consider some variants of $B_{\epsilon}$ and $S_{\epsilon}'$.
Define $n(\epsilon)=\inf\{n: T_n^0 \geq \epsilon^{-a}\}$, where
$a \in (0,2/3)$, as explained later,
and $\bar B_{\epsilon}$ is defined as
\begin{equation}
\bar B_{\epsilon}(t)= 2 \epsilon \ (t-T_{n(\epsilon)}^0)
1_{t\geq T_{n(\epsilon)}^0 }+B(t) \, ;
\end{equation}
this is the diffusion obtained by turning on a small drift of magnitude
$ 2 \epsilon$ after the stopping time $T_{n(\epsilon)}^0$. In exactly the
same way, we define $\bar T_i$ and $\bar \tau_i$ for $\bar B_{\epsilon}$
(but without the $\epsilon$ superscript) and define $\bar S_{\epsilon}(n)
\equiv \bar S_{\epsilon}^a(n)=\bar B_{\epsilon}(\bar T_n)$. In particular
$(\bar T_i,\bar \tau_i)$ and $(T_i^0,\tau_i^0)$ coincide
for $i\leq n(\epsilon)$.

Proposition~\ref{principal} is an immediate consequence of the next two lemmas.
The first relates $\bar S_{\epsilon}$ and the continuous process
$B_{\epsilon}$.
\begin{lemma}
\label{firstlemma}
There exists $C>0$ such that for any $l'\in(0,1)$,
\begin{equation}
\label{kdv1}
\mathbb{P}(\forall n\in\mathbb{N}, \, \ \bar S_{\epsilon}(n)
\geq-1-l' K \sqrt{n})\leq C \ \mathbb{P}(\forall t\in\mathbb{R}^+, \,
B_{\epsilon}(t) \geq-3- K \sqrt{t} )+O(\epsilon) \, .
\end{equation}
\end{lemma}

The next lemma relates $S_{\epsilon}'$ and $\bar S_{\epsilon}$.

\begin{lemma}
\label{secondlemma}
Let $S_{\epsilon}'$
and $\bar S_{\epsilon}^a\equiv\bar S_{\epsilon}$ be as defined above.
There exists $0<a<1$ such that for any $l\in(0,1)$,
\begin{equation}
\label{kdv}
\mathbb{P}(\forall n\in\mathbb{N}, \, \ S_{\epsilon}'(n) \geq -1-l K
\sqrt{n})\leq \mathbb{P}(\forall n\in\mathbb{N}, \, \ \bar S_{\epsilon}(n)
\geq -1-K \sqrt{n} ) +O(\epsilon) \, .
\end{equation}
\end{lemma}

\noindent {\it Proof of Lemma~\ref{firstlemma}.\/}
$\bar B_{\epsilon}$ has a smaller positive drift than $B_{\epsilon}$ and
therefore it is enough to prove (\ref{kdv1}) with $B_{\epsilon}$
replaced by $\bar B_{\epsilon}$.


By construction, for $t \in [\bar T_n , \bar T_{n+1})$,
$|\bar B_{\epsilon} (t) -\bar S_{\epsilon}(n)| < 1$, implying that
\begin{align}
\label{fff}
\mathbb{P}\left(\forall t, \, \bar B_{\epsilon}(t)
\geq-3-K \sqrt{t} \right) \geq
\mathbb{P}\left(\forall n, \,  \bar S_{\epsilon}(n) \geq-1-K
(\bar T_{n})^\frac{1}{2}\right) \nonumber \\
\geq \mathbb{P}\left(\forall n, \, \bar S_{\epsilon}(n) \geq-1-K
(\bar T_{n})^\frac{1}{2} \ {\textrm {and}}  \ \forall n\leq n(\epsilon), \
\bar T_{n} \geq n \ l \right) \, .
\end{align}
Here $l$ is arbitrary in $(0,1)$. To conclude the argument, we
proceed in two parts.
\begin{enumerate}
\item The first part is to show that except on a set of probability
$O(\epsilon)$, $K(\bar T_{n})^{1/2}$
can be replaced by $K (l' n)^{1/2}$ in the last expression of~(\ref{fff}),
with $l'=l / (2-l)$ so that $l' \to 1$ as $l \to 1$. This will be done
essentially by an application of the Law of Large Numbers.
\item Once the above replacement has been made, the desired conclusion
follows directly from the correlation inequality of Lemma~\ref{third}
and the inequality,
\begin{equation}
\mathbb{P}(\forall n\leq n(\epsilon), \, \ {\bar T_n} \equiv
\sum_{i=1}^n \bar \tau_i \geq n \ l)\geq \mathbb{P}(\forall n, \, \ \
\sum_{i=1}^n (\tau_i^0-l) \geq 0) >0 \, ,
\end{equation}
where the $\tau_i^0$ were defined in~(\ref{tautau}), with $\epsilon=0$.
Noting that $\mathbb{E} (\tau_i^0) = 1$ and hence
$\mathbb{E} (\tau_i^0 - l) > 0$, the last displayed inequality is a standard
fact about sums of i.i.d. positive mean random variables. In gambling terms,
it  says that a gambler with a slight advantage has a strictly positive
probability of never falling behind.
\end{enumerate}

It remains to justify the first part, for which it is enough to prove that,
up to an error of at most $O(\epsilon)$,  on the event
$\{\forall n\leq n(\epsilon), \, \ {\bar T_n} \equiv
\sum_{i=1}^n \bar \tau_i\geq n \ l\}$, the inequality
$\sum_{i=1}^n \bar \tau_i> l' n$ is actually valid for all $n$.
For $n>n(\epsilon)$, the $\bar \tau_i$ are the exit times
$\tau_i^\epsilon$ of Brownian motion with a small drift~$\epsilon$.
Clearly, $\mathbb{E} (\tau_i^\epsilon) \to \mathbb{E} (\tau_i^0) = 1$
as $\epsilon \to 0$. By the Law of Large Numbers and standard large
deviation estimates, we can assume  that $n(\epsilon)$ is in
$[l \, \epsilon^{-a}, (1/l) \, \epsilon^{-a}]$ and
show that the event
\begin{equation}
\label{eve2}
\{\forall n \geq n(\epsilon)+(1-l)n(\epsilon), \,
\sum_{i=n(\epsilon)+1}^n \bar \tau_i \geq l(n-n(\epsilon))\}
\end{equation}
occurs, except on a set of probability $O(\epsilon)$. Hence, up to this error,
on the event
$\{\forall n\leq n(\epsilon), \, \ \ \sum_{i=1}^n \bar \tau_i \geq n \ l\}$,
the inequality $\sum_{i=1}^n \bar \tau_i \geq n \ l$  can be extended from
all $n \leq n(\epsilon)$ also to all $n>n(\epsilon)+(1-l)n(\epsilon)$. Hence,
it only remains
to control the indices $n$ in $\{n(\epsilon)+1,...,n(\epsilon)+
[(1-l)n(\epsilon)]\}$. Since $\bar \tau_1+...+\bar \tau_{n(\epsilon)} \geq
n(\epsilon) l $, we get that for any such $n$,
\begin{equation}
\bar \tau_{1}+...+\bar \tau_n \geq n(\epsilon) l = l \ n
\frac{n(\epsilon)}{n} \geq \frac{l}{1+(1-l)} n=l' n \, .
\end{equation}
This completes the proof of Lemma~\ref{firstlemma}

\begin{lemma}
\label{third}
\begin{align*}
\mathbb{P}(\forall n, \, \ \bar S_\epsilon(n) \geq -1-K
\sqrt{l'} n^\frac{1}{2} \ {\textrm{and}}  \
\forall n\leq n(\epsilon), \ \sum_{i=1}^n\bar \tau_i\geq n \ l) \\
\geq\mathbb{P}(\forall n, \, \ \bar S_\epsilon (n) \geq -1-K \sqrt{l'}
n^\frac{1}{2})\ \ \mathbb{P}(\ \forall n\leq n(\epsilon),
\ \sum_{i=1}^n\bar \tau_i\geq n \ l).
\end{align*}
\end{lemma}

\noindent {\it Proof of Lemma~\ref{third}.}
This result is a consequence of the FKG inequality for independent
random variables. The variables
$\{\tau_i^0 \equiv T_i^0 - T_{i-1}^0 \}$ (see~(\ref{tautau})) and $S_0= S'_0$
are completely independent, since in the case of a standard Brownian motion,
knowing the exit time from the interval $[-1,1]$ does not give any information
about the exit location. Hence, given $n(\epsilon)$, $\bar S_{\epsilon} (n)$
behaves as a usual symmetric simple random walk for $n \leq n(\epsilon)$.
Thereafter the walk has a small positive drift to the right.
This suggests that the indicator of the event
$\{\forall n, \bar S_\epsilon (n) \geq -1-K \sqrt{l'} n^\frac{1}{2}\}$
can be expressed as a nondecreasing function of  $\{\tau_{i}^0\}$
(and some other variables to be determined) since
the larger $\{\tau_{i}^0\}$ is, the smaller $n(\epsilon)$ will be,
inducing more drift for $\bar S_{\epsilon}$.
To make this more precise,
we will couple $S'_0, S'_\epsilon$ and $\bar S_\epsilon$.

The coupling involves the mutually independent $(0,\infty)$-valued
$\{\tau_{i}^0\}$, $\{-1,+1\}$-valued $\{S'_0(i)-S'_0(i-1)\}$
and $\{0,1\}$-valued $\{X_i(\epsilon)\}$, with
\begin{equation}
\mathbb{P}(X_i(\epsilon)=1) = \mathbb{E}(S'_\epsilon(i) - S'_\epsilon(i-1))/2
\ = \ \epsilon+o(\epsilon).
\end{equation}
The coupling is not via a Brownian motion but rather is given in
terms of our independent variables by
\begin{equation}
S_{\epsilon}'(n) \ = \ S'_0(n) + 2 \sum_{i=0}^n X_i(\epsilon)
\end{equation}
and
\begin{equation}
\bar S_{\epsilon}(n) \ = \ S'_0(n) + 2 \sum_{i=0}^n
1_{n > n(\epsilon)}X_i(\epsilon) \, .
\end{equation}
It is clear now that the above suggestion about the nondecreasing
nature of the event in question is indeed valid. Since the other 
event, $\{\forall n\leq n(\epsilon), \ \sum_{i=1}^n\bar \tau_i\geq n \ l\}$
is clearly nondecreasing with
$\tau_{i}^0$, the claim of Lemma~\ref{third} follows by the FKG inequality.

\noindent {\it Proof of Lemma~\ref{secondlemma}.\/}
Let us condition on $n(\epsilon)$.
We use the coupling of $\bar S_{\epsilon}=\bar S_{\epsilon}^a$ and
$S_{\epsilon}'$ just discussed. Now
\begin{align}
\label{reduce}
\mathbb{P}(\forall n\in\mathbb{N}, \, \ S_{\epsilon}'(n) \geq -1-l K \sqrt{n})
\leq \mathbb{P}(\forall n\in\mathbb{N}, \, \ \bar S_{\epsilon}(n) \geq
-1-K \sqrt{n} )  \nonumber \\ 
+ \mathbb{P}(\forall n\in\mathbb{N}, \, \  S_{\epsilon}'(n) \geq
-1- l K \sqrt{n} \ {\textrm {and}} \
\bar S_{\epsilon}(\cdot) \ \textrm{hits} \ \ -1-K \sqrt{n}   ) 
\end{align}
and we need to prove that the last term is of order $\epsilon$ for a
suitable choice of the exponent $a$ (where $\epsilon^{-a}$
is the time threshold at which $\bar S_{\epsilon}$
starts drifting).

First,
\begin{eqnarray*}
\mathbb{P}(\forall n, \, \  S_{\epsilon}'(n) \geq -1- l K \sqrt{n} \
{\textrm {and}} \
\bar S_{\epsilon} \ \textrm{hits} \ -1-K \sqrt{n}   )
\leq \\
\mathbb{P}({\textrm {for some}}\,n, \
S_{\epsilon}'(n)-\bar S_{\epsilon}(n)
\geq K(1-l)\sqrt{n}) \, .
\end{eqnarray*}
Denoting by $k_i$ the solution of $i=K(1-l)\sqrt{k_i}$, we will show by
induction on $i$ that as $\epsilon \to 0$,
\begin{align}
\label{ps}
\mathbb{P}({\textrm {for some}}\ n\in (0,k_i],
S_{\epsilon}'(n)-\bar S_{\epsilon}(n)
\geq  K(1-l)\sqrt{n})  \, = \, \nonumber \\
O(\epsilon k_1 + \epsilon^2 (k_2-k_1)^2 +...+ \epsilon^2 (k_{i}-k_{i-1})^2).
\end{align}
Assuming this has been proved, we let $N$ be such that
$k_{N}\leq {n(\epsilon)}{\leq k_{N-1}}$, and (\ref{ps}) then implies that
\begin{eqnarray*}
\mathbb{P}({\textrm {for some}}\ n\in (0,n(\epsilon)], \, S_{\epsilon}'(n)-
\bar S_{\epsilon}(n)\geq K(1-l) n^\frac{1}{2})
\, =  \nonumber \\ 
 O(\epsilon k_1 + \epsilon^2 (k_2-k_1)^2 +...+
\epsilon^2 (k_{N}-k_{N-1})^2) \, .
\end{eqnarray*}
For large $n$, $k_{n+1}-k_{n} \approx  \frac{2 n}{(1-l)^2 K^2}$
and therefore $\sum_{i=1}^N (k_{i}-k_{i-1})^2=O(N^{3})=
O(n(\epsilon)^{\frac{3}{2}})$, implying that
\begin{equation}
\mathbb{P}({\textrm {for some}}\ n\in (0,n(\epsilon)], \, S_{\epsilon}'(n)-
\bar S_{\epsilon}(n)\geq K(1-l) n^\frac{1}{2})=
O(\epsilon^2 n(\epsilon)^{\frac{3}{2}}) \, .
\end{equation}
Since we can assume by the Law of Large Numbers that ${l}
\epsilon^{-a} \leq n(\epsilon)\leq \epsilon^{-a}/l $,
taking $a<\frac{2}{3}$
implies that $\epsilon^2 n(\epsilon)^{\frac{3}{2}}=O(\epsilon)$, and
we get that the last term of  (\ref{reduce}) is $O(\epsilon)$. Note that
everything was done independently of $n(\epsilon)$ (except that
$l \epsilon^{-a} \leq n(\epsilon)\leq \epsilon^{-a}/l$). Therefore,
summing over the possible values of $n(\epsilon)$
would finish the proof.

It remains to prove~(\ref{ps}), which we do by induction.
First, for $i=1$, since $1=K(1-l) \sqrt{k_1}$,
\begin{eqnarray*}
\mathbb{P}({\textrm {for some}}\ n\in (0,k_1],
\, S_{\epsilon}'(n)-\bar S_{\epsilon}(n)
\geq K(1-l) \sqrt{n})
& = & \mathbb{P}(S_{\epsilon}'(\cdot)-\bar S_{\epsilon}(\cdot) \
\textrm{jumps on
$[0,k_1]$}) \\ 
& = & k_1 O(\epsilon)  .
\end{eqnarray*}
Next, assuming that~(\ref{ps}) is valid up to $i$, we have
\begin{eqnarray*}
& \mathbb{P}({\textrm {for some}}\ n\in(0,k_{i+1}], \,
\ S_{\epsilon}'(n)-\bar S_{\epsilon}(n)
\geq K(1-l)\sqrt{n}) & \\
& \leq  \mathbb{P}( {\textrm {for some}}\ n\in(0,k_i], \  S_{\epsilon}'(n)-
\bar S_{\epsilon}(n)\geq K(1-l)\sqrt{n}) & \\
& +
\mathbb{P}(\forall n \in (0,k_i], \ S_{\epsilon}'(n)-
\bar S_{\epsilon}(n)<K(1-l)\sqrt{n} & \\
&\ {\textrm {and for some}} \
n \in (k_{i},k_{i+1}], \ S_{\epsilon}'(n)-
\bar S_{\epsilon}(n)\geq K(1-l)\sqrt{n}) \, . &
\end{eqnarray*}
We need to bound the last term of this inequality.
Since on  $(k_i,k_{i+1}]$ we have $(1-l)K\sqrt{n}\in(i,i+1]$
(by the definition of $k_i$), and since $S_{\epsilon}'-\bar S_{\epsilon}$
only takes integer value, if
$S_{\epsilon}'(n)-\bar S_{\epsilon}(n)\geq K(1-l)\sqrt{n}$, then
$S_{\epsilon}'(k_{i+1})-\bar S_{\epsilon}(k_{i+1})\geq i+1$. On the
other hand, $S_{\epsilon}'(k_i)-\bar S_{\epsilon}(k_i) < i$,
and then our process has to jump at
least twice on $(k_i,k_{i+1}]$. But this probability is bounded by
a term of order $(k_{i+1}-k_{i})^2\epsilon^2$ and (\ref{ps}) follows.
This completes the proof of Lemma~\ref{secondlemma}.


\bigskip

\bigskip

{\em Acknowledgements.}
The research of L.R.G.~Fontes was supported in part by FAPESP grant
2004/07276-2 and CNPq grants 307978/2004-4 and 484351/2006-0;
the research 
of the other authors was supported in part by N.S.F. grants DMS-01-04278 and 
DMS-06-06696.


\begin{thebibliography}{SSS}
\bibitem[A81]{A81}
R.~Arratia, Coalescing Brownian motions and the voter model on ${\mathbb Z}$, 
Unpublished partial manuscript (circa 1981),
available from rarratia@math.usc.edu.

\bibitem[AH06]{AH06}
G.~Amir, C.~Hoffman,
A special set of exceptional times for dynamical random walk on ${\mathbb Z}^2$.
{\em Arxiv: math.PR/0609267}.

\bibitem[BHPS03]{BHPS03}
I.~Benjamini, O.~H\"{a}ggstrom, Y.~Peres, J.~E.~Steif.
Which properties of a random sequence are dynamically sensitive?
{\em Ann. Probab.}~{\bf 31}, (2003), 1--34.

\bibitem[CFN06]{CFN06}
F.~Camia, L.R.G.~Fontes, C.M.~Newman.
The scaling limit geometry of near-critical 2D percolation.
{\em J. Stat. Phys.}~{\bf 125}, (2006), 1155-1171.

\bibitem[F73]{F73}
D.F.~Fraser.
The rate of convergence of a random walk to Brownian Motion.
{\em Ann. Probab.}~{\bf 4}, (1973), 699--701.

\bibitem[FINR04]{FINR04}
L.R.G.~Fontes, M.~Isopi, C.M.~Newman, K.~Ravishankar.
The Brownian web: characterization and convergence.
{\em Ann. Probab.}~{\bf 32}, (2004), 2857--2883.

\bibitem[FINR05]{FINR05}
L.R.G.~Fontes, M.~Isopi, C.M.~Newman, K.~Ravishankar.
Coarsening, nucleation, and the marked brownian web.
{\em  Ann. Inst. H. Poincar\'{e}, Probab. et Stat.}~{\bf 42}, (2006), 37-60.

\bibitem[G89]{G89}
G.~Grimmett, {\em Percolation}, Springer-Verlag, 1989

\bibitem[HPS97]{HPS97}
O.~H\"{a}ggstr\"{o}m, Y.~Peres, J.~Steif.
Dynamical percolation.
{\em Ann. Inst. H. Poincar\'{e}, Probab. et Stat.}~{\bf 33}, (1997), 497-528.

\bibitem[H78]{H78}
T.E.~Harris, Additive set-valued Markov Processes and graphical methods, 
{\em Ann. Probability }~{\bf 6}, 355-378 (1978)

\bibitem[Hoff05]{Hoff05}
C.~Hoffman.
Recurrence of simple random walks on ${\mathbb Z}^2$ is dynamically sensitive.
{\em ALEA}~{\bf 1}, (2006), 35-45.

\bibitem[HW07]{HW07}
C.~Howitt, J.~Warren.
Dynamics for the Brownian web and the erosion flow.
{\em Arxiv: math.PR/0702542}.

\bibitem[NRS07]{NRS07}
C.M.~Newman, K.~Ravishankar, E.~Schertzer, in prep.

\bibitem[S77]{S77}
S.~Sato.
Evaluation of the first-passage time probability to a square root boundary for the Wiener process.
{\em J. Appl. Probab.}~{\bf 14}, (1977), 850--856.

\bibitem[SS05]{SS05}
O.~Schramm, J~Steif.
Quantitative noise sensitivity and exceptional times for percolation.
{\em ArXiv: math.PR/0504586}.

\bibitem[STW00]{STW00}
F.~Soucaliuc, B.~T\'oth, W.~Werner, Reflection and coalescence between
independent one-dimensional Brownian paths,
{\em Ann. Inst. H. Poincar\'{e}, Probab. et Stat.}~{\bf 36}, (2000), 509--545.

\bibitem[SS06]{SS06}
R.~Sun, J.M.~Swart.
The Brownian net.
{\em ArXiv: math.PR/0610625}.


\bibitem[TW98]{TW98}
B.~T\'oth, W.~Werner.
The true self-repelling motion.
{\em Probab.\ Th. Rel. Fields}~{\bf 111}, (1998), 375--452.

\end{thebibliography}
\end{document}